%% file: main.tex
\documentclass{IEEEtran}

\usepackage{amsmath,amssymb,amsfonts,amsthm}
\usepackage{cite}
\usepackage{algorithmic}
\usepackage{mathrsfs}
\usepackage{graphicx}
\usepackage{algorithm,algorithmic}
\usepackage{textcomp}
\usepackage{epstopdf}

\usepackage[capitalise]{cleveref}

\newtheorem{definition}{Definition}
\newtheorem{assumption}{Assumption}
\newtheorem{condition}{Condition}
\newtheorem{lemma}{Lemma}
\newtheorem{corollary}{Corollary}

\newtheorem{theorem}{Theorem}
\newtheorem*{remark}{Remark}

\crefformat{equation}{\textup{#2(#1)#3}}

\newcommand{\bigroman}[1]{\uppercase\expandafter{\romannumeral#1}}

\def\BibTeX{{\rm B\kern-.05em{\sc i\kern-.025em b}\kern-.08em
    T\kern-.1667em\lower.7ex\hbox{E}\kern-.125emX}}
\begin{document}

\title{Constrained Optimization with Compressed Gradients: A Dynamical Systems Perspective}

\author{Zhaoyue Xia, \IEEEmembership{Student Member, IEEE},
        Jun Du, \IEEEmembership{Senior Member, IEEE}, Chunxiao Jiang, \IEEEmembership{Fellow, IEEE}, \\ H. Vincent Poor, \IEEEmembership{Life Fellow, IEEE}, 
        %Zhu Han, \IEEEmembership{Fellow, IEEE}, 
        and Yong Ren, \IEEEmembership{Senior Member, IEEE}
         % <-this % stops a space
        %Lajos~Hanzo,~\IEEEmembership{Life~Fellow,~IEEE}
\thanks{Z. Xia, J. Du, and Y. Ren are with the Department of Electronic Engineering, Tsinghua University, Beijing, 100084, China, (e-mail: xiazy19@mails.tsinghua.edu.cn; \{jundu,reny\}@tsinghua.edu.cn).} % <-this % stops a space
\thanks{C. Jiang is with Tsinghua Space Center, Tsinghua University, Beijing, 100084, China. (e-mail: jchx@tsinghua.edu.cn).}
\thanks{H. V. Poor is with the Department of Electrical and Computer Engineering, Princeton University, Princeton, NJ 08544 USA, (e-mail: poor@princeton.edu).}
% \thanks{Z. Han is with the Department of Electrical and Computer Engineering at the University of Houston, Houston, TX 77004 USA, and also with the Department of Computer Science and Engineering, Kyung Hee University, Seoul, South Korea, 446-701, (e-mail: hanzhu22@gmail.com).}
}

\maketitle

\begin{abstract}
Gradient compression is of growing interests for solving constrained optimization problems including compressed sensing, noisy recovery and matrix completion under limited communication resources and storage costs. Convergence analysis of these methods from the dynamical systems viewpoint has attracted considerable attention because it provides a geometric demonstration towards the shadowing trajectory of a numerical scheme. In this work, we establish a tight connection between a continuous-time nonsmooth dynamical system called a perturbed sweeping process (PSP) and a projected scheme with compressed gradients. Theoretical results are obtained by analyzing the asymptotic pseudo trajectory of a PSP. We show that under mild assumptions a projected scheme converges to an internally chain transitive invariant set of the corresponding PSP. Furthermore, given the existence of a Lyapunov function $V$ with respect to a set $\Lambda$, convergence to $\Lambda$ can be established if $V(\Lambda)$ has an empty interior. Based on these theoretical results, we are able to provide a useful framework for convergence analysis of projected methods with compressed gradients. Moreover, we propose a provably convergent distributed compressed gradient descent algorithm for distributed nonconvex optimization. Finally, numerical simulations are conducted to confirm the validity of theoretical analysis and the effectiveness of the proposed algorithm. 
\end{abstract}

\begin{IEEEkeywords}
Constrained compressed optimization, dynamical system, convergence analysis, low-bit signal processing.
\end{IEEEkeywords}

\input{sections/introduction.tex}
\input{sections/notation.tex}
\input{sections/continuous-to-cascades.tex}
\input{sections/application.tex}
\input{sections/simulation.tex}
\input{sections/conclusion.tex}

\bibliographystyle{IEEEtran}
\bibliography{IEEEabrv,ref}

\end{document}

%% file: sections/introduction.tex
\section{Introduction}
\label{sec:introduction}

Constrained optimization is a fundamental problem in mathematical programming \cite{2023LocalLinearConv,2024VariancedReduced,2023StochasticCompositionalGradient}, where the objective is to minimize a function subject to a set of constraints. These constraints are usually nonlinear and they often reflect real-world limitations such as resource availability, physical laws, or operational boundaries. The complexity of constrained optimization stems from the interplay between the objective function and these constraints, enforcing the trajectory of iterations produced by a optimization scheme to move along the boundaries of a constrained set or within the set. 

In popular machine learning applications (e.g., federated learning, neural network quantization, decentralized gradient tracking, etc.), compressed gradients instead of exact inputs are used in consideration of privacy concerns, transmission overheads and storage costs. In \cite{2024ProjGD}, the authors proposed projected gradient descent (GD) method for spectral compressed sensing. For transmission overheads in federated learning, a compressed stochastic GD (SGD) with adaptive step sizes was proposed \cite{2024Adaptive}. The authors \cite{2021CompressedGD} proposed a compressed GD algorithm with Hessian-aided error compensation.

% From a geometric viewpoint, given a convex set $C \subseteq \mathbb{R}^m$ and a point $x \in C$, any vector $v \in \mathbb{R}^m$ can be decomposed into two orthogonal components in the tangent cone $\mathcal{T}_C(x)$ and the normal cone $N_C(x)$, respectively. Moreover, the projection into the set $C$ actually means canceling the component in the normal cone. Therefore, any movement in free space will be projected into the tangent cone. 

% \vspace{-6pt}
% \subsection{Background and Motivation}
% A constrained optimization problem is to solve
% \begin{equation}
%   \min_{x\in C} f(x) = \min_{x\in \mathbb{R}^m} f(x) + \delta_C(x),
% \end{equation}
% where $\delta_C(\cdot)$ is the indicator function of set $C \subseteq \mathbb{R}^m$ and $f: \mathbb{R}^m \to \mathbb{R}$ is a nonconvex function. The proximal gradient method, an acknowledged algorithm to solve constrained optimization problems, 
As a classical and significant topic, convergence analysis of constrained optimization methods has been of interest due to its essential differences from that of unconstrained schemes. To be specific, a constrained method iteratively seeks a proximal point within a set, which yields a nonsmooth part in the iteration. To analyze the convergence properties of an optimization scheme, two principal methodologies have emerged: numerical analysis and dynamical systems theory. Numerical analysis \cite{2023QuantizationforDecentralized,2024PairwiseConstraints,2022SampleBased} offers a straightforward depiction of the concrete convergence rates, providing a clear understanding of the speed at which an algorithm approaches its optimal solution. However, this approach often lacks the deeper geometric insights that can be gleaned from a dynamical systems perspective. This latter approach, grounded in the study of continuous-time systems, enriches the analysis by revealing the underlying geometric structures and dynamics that influence the convergence behavior of optimization schemes. 

For constrained optmization, we take the standard projected SGD method for example:
\begin{equation}
  x_{k+1} = P_C[x_k - \alpha_k (\nabla f(x_k) + \xi_k)],
\end{equation}
where $C$ is a convex subset in $\mathbb{R}^m$ and $\xi_k$ is a random perturbation. Different from the unconstrained scheme, a projector is required to ensure that $x_k$ remains in $C$. Recall that an unconstrained scheme is linked with the continuous-time dynamical system $dx/dt = -\nabla f(x)$ \cite{1996DynSystApprochStocApprox}. Likewise, we are interested in the following differential inclusion:
\begin{equation}
  \frac{dx}{dt} \in -\nabla f(x(t)) - N_C(x(t)),
\end{equation}
where $N_C(x)$ is the normal cone of $C$ at $x$. The most significant advantage of a dynamical systems perspective lies in the simplicity of the treatment of a continuous-time system. Moreover, the convexity of $f$ is not required to establish the connection between a continuous dynamical system and the discrete iterative method. Therefore, we can focus on the limiting behavior of the continuous dynamical system. 

In a pratical system, the gradient measurements can be compressed for low storage costs and low hardware complexity especially in current machine learning applications. In this case, $\phi(\nabla f(x))$ is used instead of $\nabla f(x)$ for a compressor $\phi$. Correspondingly, the compression error $\phi(\nabla f(x)) - \nabla f(x)$ can be treated as a random perturbation. Therefore, it is important to investigate the effects of random perturbations and establish the connection between a perturbed iterative method and a continuous dynamical system.

In recent years, there has been a significant surge in research on constrained optimization from a dynamical systems perspective. We present a synthesis of some of the most recent and representative findings in this domain. In \cite{2023DSADMM}, the authors analyzed differential inclusions associated with accelerated variants of the alternating direction method of multipliers (ADMM) and illustrated a tradeoff between the convergence rate and the damping factor. A primal-dual dynamical system approach was proposed to track an inequality constrained time-varying convex optimization problem in \cite{2024DSTV}. For online time-varying optimization of linear time-invariant systems, a linear dynamical system was applied to develop a convergent projected primal-dual gradient flow method \cite{2022TVOLti}. Accelerated methods were developed under the framework of fixed-time stability of nonlinear dynamical systems for functions under Polyak-Ljasiewicz inequality conditions in \cite{2024FixedTime}. In \cite{2023DynamicOptimComplementarity}, dynamic optimization theory was established for nonlinear complementarity systems. The second-order dynamical system was extended to constrained distributed optimization in \cite{2024SecondOrderPrimalDual}. 

According to the classical result \cite{1996DynSystApprochStocApprox} in stochastic approximation, the continuous-time dynamics of an unconstrained iterative discrete method can be demonstrated by an ordinary differential equation (ODE). To be specific, a stochastic approximation scheme given by
\begin{equation}
  \label{eqn:general-stochastic-approximation}
  x_{n+1} = x_n + \alpha_{n+1} (\psi(x_n) + \xi_{n+1}),
\end{equation}
converges to an internally chain transitive set of the dynamical system expressed by the ODE $\dot{x} = \psi(x)$, where $\dot{x}$ means the derivative with respect to time, $\{\alpha_n\}$ are vanishing step sizes, $\psi$ is Lipschitz continuous, and $\{\xi_n\}$ is a sequence of martingale difference noise. 

Although it is straightforward to show that a GD method converges simply by replacing $\psi$ by $-\nabla f$ in \cref{eqn:general-stochastic-approximation}, the underlying relationship between the internally chain transitive set and the critical point set of $f$ is not immediately apparent. Bridging this gap is the concept of a Lyapunov function $V$ (total energy of the system), which plays a pivotal role in the stability analysis of dynamical systems. By incorporating the objective function $f$ into the Lyapunov function, the substantial dissipation of energy leads to a local minimum of $V$ and hence $f$. 
 
As a counterpart to the GD method, the projected gradient method is intrinsically linked to a PSP with a constraint set\footnote{Note that since the constraint set is time-independent for a standard constrained optimization problem, we will restrict our discussion to a PSP with a fixed set.}, as established in \cite{2016PerturbedSweepingProcess}. However, it is not evident whether the convergence conclusions drawn for unconstrained stochastic approximations remain valid in the context of constrained problems. In fact, this uncertainty arises from the nonsmooth characteristics inherent to projected gradient methods for constrained optimization. 

Furthermore, there is a natural inclination to employ a nonsmooth Lyapunov function that encapsulates the complexity of the problem. Ideally, such a function would be decomposed into two components: a smooth part that corresponds to the vector field $\psi$, and a lower-semicontinuous part that accounts for the constraints. Unfortunately, this approach often encounters difficulties, as the nonsmoothness can impede the straightforward application of traditional Lyapunov theory. 

In fact, optimal control of a PSP has been a well-studied problem, which comes from the application to the crowd motion model. A number of theoretical results have been developed \cite{2019OptimalcontrolPSP,2022PSPNonsmooth,2022GlobalAsympPSP}. Numerical analysis \cite{2011NumericalSchemeforPSP,2013ConvergenceOrder} on discretization of a continuous-time PSP is aimed at deriving the convergence order of the numerical scheme towards the continuous dynamics within a finite time. These results, however, do not provide the Lyapunov properties of $\omega$-limit sets of a PSP and fail in the infinite-time asymptotic analysis.

In this article, we develop dynamical systems theory with respect to constrained optimization schemes, aimed at providing a general framework for convergence analsyis. Specifically, the contributions of this work can be summarized as follows:
\begin{itemize}
  \item We provide a Lyapunov analysis for a PSP with a fixed constraint set. We show that if a PSP is a gradient-like dynamical system with a compact convex set, the $\omega$-limit set of any initial point $x$ is contained in the fixed point set of the corresponding Lyapunov function. 
  \item We establish the connection between a PSP and its Euler discretization and show that the discrete iterations converge to an internally chain transitive set of the PSP, which is similar to the behavior of unconstrained stochastic approximation. Furthermore, we develop the Lyapunov theory for such an iterative method. 
  \item By utilizing the theory of Lyapunov pairs, we provide several examples of convergence analysis of projected variants of popular gradient-based methods. Based on the established theoretical results, we develop a provably convergent distributed projected compressed gradient descent scheme for distributed nonconvex optimization. 
  \item Numerical simulations are conducted to verify the validity of the theoretical analysis. Results show that the projected algorithms (including the distributed scheme) succeed in converging to local minima within the constraint set.
\end{itemize}

The rest of the article is organized as follows. Basic concepts and notation are introduced in \cref{sec:Basic-concepts}. Subsequently, the primary theoretical results are demonstrated and derived in \cref{sec:analysis}. We provide examples of applications to optimization in \cref{sec:appl}. Numerical simulation results are presented in \cref{sec:simulation} and \cref{sec:conclusion} concludes this article.

%% file: sections/notation.tex
\vspace{-6pt}
\section{Basic Concepts and Notation}
\label{sec:Basic-concepts}

In this section, we provide some notation and basic concepts (especially in the theory of dynamical systems) to be used throughout the article.

Let $X$ be a topological space, $\mathbb{R}^+$ be the semigroup of nonnegative real numbers and $\mathbb{T} \subseteq \mathbb{R}^+$ be a subsemigroup of the additive group. A triplet $(X,\mathbb{T},\pi)$, where $\pi: \mathbb{T}\times X \to X$ is a continuous mapping satisfying $\pi(0,x)=x$ and $\pi(s,\pi(t,x)) = \pi(s+t,x)$ for all $x\in X$ and $s,t\in \mathbb{T}$, is called a (continuous) dynamical system. Given $x \in X$, the set $\Upsilon_x := \pi(\mathbb{T}, x)$ is called a trajectory (associated with $x$). A point $x\in X$ is called a fixed point of $(X,\mathbb{T},\pi)$ if $\pi(t,x)=x$ for all $t\in\mathbb{T}$. A discrete dynamical system where $\mathbb{T} \subseteq \mathbb{Z}$ is called a cascade. 

A nonempty set $M \subseteq X$ is called (positively) invariant with respect to a dynamical system $(X,\mathbb{T},\pi)$ if $\pi(t, M) \subset M$ for every ($t \ge 0$) $t \in \mathbb{T}$. Let $J \subseteq X$. The set 
\begin{equation}
  \omega(J) := \bigcap_{t\ge 0} \overline{\bigcup_{s\ge t} \pi(s,J)},
\end{equation}
where $\bar{A}$ denotes the closure of a set $A$, is called the $\omega$-limit set for $J$. An equivalent definition of the $\omega$-limit set is 
\begin{equation}
  \omega(J) = \{u\in X: \exists x\in J, \exists t_n \to \infty, \pi(t_n,x)\to u \}.
\end{equation}

Let $\varSigma \subseteq X$ be a compact positively invariant subset of a metric space $(X,d)$, $\varepsilon > 0$, and $t > 0$. The collection $\{x = x_0, x_1, x_2, \dots, x_k = y;t_0, t_1,\dots,t_k\}$ of points $x_i \in \varSigma$ and the numbers $t_i \in \mathbb{T}$ such that $t_i \ge t$ and the distance $d(\pi(t_i,x_i), x_{i+1}) < \varepsilon$, $(i = 0, 1,\dots,k-1)$ is called an $(\varepsilon, t, \pi)$-chain joining the points $x$ and $y$. The set $\varSigma$ is called \emph{internally chain transitive} if for all $a,b \in \varSigma$, $\varepsilon > 0$ and $t > 0$,  there exists an $(\varepsilon, t, \pi)$-chain in $\varSigma$ connecting $a$ and $b$.

A dynamical system $(X, \mathbb{T}, \pi)$ is said to be a gradient-like dynamical system if it has a global Lyapunov function $V: X \to \mathbb{R}$, i.e., $V$ is continuous and satisfies $V(\pi(t,x)) \le V(x)$ for all $x \in X$ and $t \in \mathbb{T}$. 

Let $S$ be a nonempty subset of a Hilbert space $\mathcal{H}$, and $x \in \mathcal{H}$. The distance between $x$ and $S$ is expressed by
\begin{equation}
  d(x;S) := \inf_{y\in S} \|x-y\|.
\end{equation}
The set of nearest points of $x$ in $S$ is defined by
\begin{equation}
  P_S(x) := \left\{u\in S: \|x-u\| = d(x;S) \right\}.
\end{equation}
For a convex subset $S \subseteq \mathcal{H}$ and $x\in \mathcal{H}$, the normal cone to $S$ at $x$ is $N_S(x) = \{v \in S: \langle v, y-x \rangle \le 0, \forall y \in S\}$. Correspondingly, we use $\mathcal{T}_S(x)$ to represent the tangent cone. Given a constrained optimization problem $\min_{x\in C} f(x)$ for a closed and convex set $C \subseteq \mathbb{R}^m$ and a differentiable function $f$, the set of Karush-Kuhn-Tucker (KKT) points is defined as $\mathcal{L}:= \{x \in C: 0 \in \nabla f(x) + N_C(x)\}$.

A sequence $\{y(t)\}_{t\in \mathbb{R}}$ of elements in $\mathcal{H}$ is said to converge to a set $J$ if $d(y(t);J) \to 0$ as $t\to +\infty$, denoted by $y(t) \to J$. Given $\lambda \in \mathbb{R}$ and a Hilbert space $(\mathcal{H},|\cdot|)$, we say that $f: \mathcal{H} \to \mathbb{R}$ is $\lambda$-convex if $f(x) - \frac{\lambda}{2} |x|^2$ is convex.

Given a nonempty set $C$, we use $\mathcal{I}_C$ to represent the indicator function of $C$, i.e., $\mathcal{I}_C(x) = 0$ if $x \in C$ and $\mathcal{I}_C(x) = +\infty$ otherwise. For a lower semi-continuous function $\varphi: \mathcal{H} \to \mathbb{R}$ on a Hilbert space $\mathcal{H}$, a vector $\xi \in \mathcal{H}$ is called a Fréchet subgradient, written $\xi \in \partial_F \varphi(x)$, at $x$ if 
\begin{equation}
  \varphi(y) \ge \varphi(x) + \langle\xi,y-x\rangle + o(\|y-x\|), \ \forall y \in \mathcal{H}.
\end{equation}
We use $\mathbb{B}(x,r)$ to denote a closed ball in a metric space centered at $x$ with radius $r$.

Let $\Upsilon(t,x)$ be a trajectory of a dynamical system and $\Lambda$ be a subset of a metric space $X$. A continuous function $V: X \to \mathbb{R}$ is called a Lyapunov function for a set $\Lambda$, if $V (y) < V (x)$ for all $x \in X \backslash \Lambda$, $y \in \Upsilon(t,x)$, $t > 0$, and $V (y) \le V (x)$ for all $x \in \Lambda$, $y \in \Upsilon(t,x)$, and $t \ge 0$. 

% A stochastic sequence $\{X_n\}_{n \in \mathbb{N}}$ is called a seqeunce of Martingale difference noise if 

Throughout the paper, two forms of the well-known Gronwall's inequalities \cite{2011FunctionalAnalysis} will be used.
\begin{itemize}
  \item \textbf{The classical differential form.} Assume that $u: [0, T) \to \mathbb{R}$ is continuously differentiable, $T\in (0, \infty)$, and satisfies the differential inequality
  \begin{equation}
    \frac{du}{dt} \le a(t) u(t) + b(t),
  \end{equation}
  for some integrable functions $a,b$ on $(0,T)$. Then, $u$ satisfies the pointwise bound 
  \begin{equation}
    u(t) \le e^{A(t)}u(0) + \int_0^t b(s) e^{A(t) - A(s)} ds,
  \end{equation}
  where $A(t) := \int_0^t a(s) ds$ for all $t \in [0,T)$. 
  \item \textbf{The discrete form.} Consider a sequence of real numbers $\{u_n\}$ such that
  \begin{equation}
    u_{n+1} \le a_{n+1}u_n + b_{n+1}, \quad \forall n \ge 0,
  \end{equation}
  where $\{a_n\}$ and $\{b_n\}$ are two given sequences of real numbers and $\{a_n\}$ is furthermore positive. Then
  \begin{equation}
    u_n \le A_n u_0 + \sum_{k=1}^n A_{k,n}b_k, \quad \forall n\ge 0,
  \end{equation}
  where $A_n := \prod_{k=1}^n a_k$, $A_{k,n}:=A_n/A_k$.
\end{itemize}

The stability analysis of the nonsmooth dynamics of a PSP naturally requires nonsmooth Lyapunov functions.

\begin{definition}[A variant of Definition 1 in \cite{2012LyapunovPair}]
  \label{definition:Lyapunov-pair}
  Let $\mathcal{H}$ be a Hilbert space. Let functions $V,W: \mathbb{R} \times \mathcal{H} \to \mathbb{R}$ be lower semi-continuous, with $W \ge 0$. We say that $(V, W)$ is a time-dependent Lyapunov pair for a dynamical system $(X, \mathbb{R}^+, \pi)$ ($X \subseteq \mathcal{H}$) if for all $x_0 \in X$ and $\forall t \ge 0$, 
  \begin{equation}
    V(t,x(t)) + \int_{0}^{t} W(\tau,x(\tau)) d\tau \le V(0,x_0),
  \end{equation}
  where $x(t) = \pi(t,x_0)$. 
\end{definition}

Identifying a suitable Lyapunov pair for nonsmooth dynamical systems is inherently complex, primarily due to the difficulty in determining the supremum of the Lie derivatives of potential Lyapunov functions. The challenge arises from the requirement to evaluate the supremum within the context of the Fréchet subdifferential, which encapsulates a broader set of candidates than the traditional derivative would allow. 

Fortunately, the following lemma provides a powerful tool to settle the problem for a PSP as \cref{eqn:sweeping-process}. 
\begin{lemma}
  \label{lemma:Lyapunov-pair}
  Let $\mathcal{H}$ be a Hilbert space. Let functions $V,W: \mathbb{R} \times \mathcal{H} \to \mathbb{R}$ be lower semi-continuous, with $W \ge 0$. $(V,W)$ is a time-dependent Lyapunov pair if and only if for all $t \ge 0$, $x \in \mathcal{H}$ and $\xi \in \partial_F V(t,x)$, we have
  \begin{equation}
    \min_{v \in N_C(x) \cap \mathbb{B}(0,\|\psi(t,x)\|)} \langle \xi, -\psi(t,x)-v \rangle + W(t,x) \le 0.
  \end{equation}
\end{lemma}
\begin{proof}
  A combination of \cite[Theorem 5.1]{2019LyapunovMonotone} and \cite[pp. 300-301, Proposition 5]{1984DifferentialInclusions}.
\end{proof}

%% file: sections/continuous-to-cascades.tex
\section{From Continuous Dynamics to Cascades}
\label{sec:analysis}

Given a compressor $\vartheta$, a projected GD algorithm with compressed gradients has a conceptual form expressed as 
\begin{equation}
  \label{eqn:compressed-form}
  z_{n+1} = P_{\mathcal{K}}[z_n - \alpha_{n+1} (\vartheta(\nabla f(z_n)) + \xi_{n})],
\end{equation}
where $\mathcal{K}$ is the constraint set and $\xi_n$ is random perturbation. Denote the compression residual error by $r_n := \vartheta(\nabla f(z_n)) - \nabla f(z_n)$. \cref{eqn:compressed-form} can be transformed into
\begin{equation}
  z_{n+1} = P_{\mathcal{K}}[z_n - \alpha_{n+1} (\nabla f(z_n) + r_n + \xi_{n})].
\end{equation}
Regarding $r_n + \xi_n$ as random perturbations, we can naturally associate the discrete evolutionary equation \cref{eqn:compressed-form} with a continuous-time constrained dynamical system
\begin{equation}
  \frac{dz}{dt} \in - \nabla f(z) - N_{\mathcal{K}}(z),
\end{equation}
in which we select $\alpha_n$ as the step size for discretization. In this article, we consider its general form as follows:
\begin{equation}
  \label{eqn:original-problem}
  \frac{dz}{dt} \in -\psi(t,z) - N_{\mathcal{K}}(z).
\end{equation}
Clearly, such a constrained non-autonomous dynamical system projects the continuous-time dynamics into $\mathcal{K}$. Moreover, this differential inclusion cannot be viewed as a variational inequality problem due to its non-autonomous nature. Indeed, the dynamics are covered by a topic termed the \emph{perturbed sweeping process}, which will be discussed in detail below.

\subsection{Results on Perturbed Sweeping Processes}
We first present sufficient conditions for the existence and uniqueness of a PSP. 
\begin{condition}[sweeping-regular]
  \label{condition:sweeping-regular}
  Let $\mathcal{H}$ and $\mathcal{F}$ be Hilbert spaces. A function $\psi: \mathbb{R}\times \mathcal{H} \to \mathcal{F}$ is said to be sweeping-regular on a pair $(I,C)$ for $I\subseteq \mathbb{R}$ and $C\subseteq \mathcal{H}$ if 
  \begin{itemize}
    \item $\forall \eta>0$, there exists an integrable nonnegative function $L_\eta(t): I \to \mathbb{R}$ such that, for all $t$ and for all $\max\{\|x\|,\|y\|\} < \eta$,
    \begin{equation}
      \|\psi(t,x) - \psi(t,y)\|_{\mathcal{F}} \le L_\eta(t) \|x-y\|_{\mathcal{H}};
    \end{equation}
    \item there exists an integrable nonnegative function $\beta: I \to \mathbb{R}$ such that, for all $t$ and for all $x \in C$, $\|\psi(t,x)\|_{\mathcal{F}} \le \beta(t)(1+\|x\|_{\mathcal{H}})$.
  \end{itemize}
\end{condition}
We begin with a useful lemma ensuring that a composite function is sweeping-regular.
\begin{lemma}
  \label{lemma:composite-sweeping-regular}
  Let $\mathcal{H}_n$ be a sequence of Hilbert spaces, $\psi_n:\mathbb{R}\times \mathcal{H} \to \mathcal{H}_n$ for $n=1,2,\dots,N$, and $\mathcal{H}=\mathcal{H}_1\times \mathcal{H}_2 \times \cdots \times \mathcal{H}_N$. If each $\psi_n$ is sweeping-regular on $(I,C)$, the composite function $\psi(t,x)=(\psi_1(t,x_1),\psi_2(t,x_2),\dots,\psi_N(t,x_N))$ is sweeping-regular on $(I,C)$, where $x=(x_1,x_2,\dots,x_N) \in \mathcal{H}$.
\end{lemma}
\begin{proof}
  The result is a straightforward consequence of the triangle inequality for Hilbert spaces.
\end{proof}

Given this condition, we have the following lemma ensuring the existence and uniqueness of a solution to \cref{eqn:original-problem}:
\begin{lemma}{\cite[Theorem 2.1]{2018GlobalStability}}
  \label{lemma:existence-solution-sweeping}
  Let $\mathcal{H}$ be a Hilbert space, $C$ be a closed and convex subset of $\mathcal{H}$, $I$ be a subset of $\mathbb{R}$, and $\psi: \mathbb{R} \times \mathcal{H} \to \mathcal{H}$ satisfying \cref{condition:sweeping-regular} for $(I,C)$. Then the PSP with $x(0) \in C$
  \begin{equation}
    \label{eqn:sweeping-process}
    -\frac{dx}{dt} \in \psi(t,x) + N_{C}(x), \ \text{a.e. }t \in I,
  \end{equation}
  has a unique absolutely continuous solution $x(t)$ defined on $I$. Moreover, for almost everywhere $t \in I$,
  \begin{equation}
    \|\dot{x}(t) + \psi(t,x(t))\| \le D \beta(t), \quad \|\psi(t,x(t))\| \le D \beta(t),
  \end{equation}
  for some constant $D = D(x(0),\int_{I} \beta(s) ds) > 0$.
\end{lemma}

Note that an absolutely continuous function $x(t)$ is said to be a solution to the sweeping process \cref{eqn:sweeping-process} on an interval $I \subseteq \mathbb{R}$ if $x(t) \in C$ for a.e. $t \in I$ and $\dot{x}(t)$ satisfies \cref{eqn:sweeping-process}. Since we will discuss properties of $\omega$-limit sets of a PSP, it is necessary to extend the solution to the entire real line $\mathbb{R}$ (or at least $\mathbb{R}^+$). By \cite[Corollary 2]{2006Equivalence}, the differential inclusion \cref{eqn:sweeping-process} is equivalent to the ODE
\begin{equation}
  \dot{x}(t) = P_{\mathcal{T}_C(x(t))} [-\psi(t,x)], \quad \text{a.e. } t\in I,
\end{equation}
where $P_{\mathcal{T}_C(x)}$ denotes the projection into the tangent cone of $C$ at $x$. By standard procedure to extend a solution of an ODE, we have the following lemma.
\begin{lemma}
  \label{lemma:global-solution-sweeping-process}
  Suppose for every $\tau > 0$, $\psi$ is sweeping-regualr on $([-\tau,\tau], C)$ or $([0,\tau], C)$. Then the solution of \cref{eqn:sweeping-process} is defined for all $t\in \mathbb{R}$ or $t\in \mathbb{R}^+$, respectively. 
\end{lemma}
\begin{proof}
  Using the bound from \cref{condition:sweeping-regular}, we have
  \begin{equation}
    \|x(t)\| \le \|x(0)\| + \int_0^t \|P_{\mathcal{T}_C(x(t))}[-\psi(s,x(s))]\| ds.
  \end{equation}
  Since projection into a closed and convex set is nonexpansive, it follows that
  \begin{equation}
    \begin{aligned}
      &\|P_{\mathcal{T}_C(x(t))}[-\psi(s,x(s))]\| \\
      &= \|P_{\mathcal{T}_C(x(t))}[-\psi(s,x(s))] - P_{\mathcal{T}_C(x(t))}[0]\| \\
      &\le \|\psi(s,x(s))\| \le \beta(s)(1+\|x(s)\|).
    \end{aligned}
  \end{equation}
  Hence we obtain
  \begin{equation}
    \|x(t)\| \le \|x(0)\| + \int_0^t \beta(s)(1+\|x(s)\|) ds.
  \end{equation}
  Using the above variant of Gronwall's inequality implies
  \begin{equation}
    \|x(t)\| \le \|x(0)\|e^{B(t)} + \int_0^t \beta(s) e^{B(t)-B(s)} ds,
  \end{equation}
  where $B(t):= \int_0^t \beta(s) ds$. By the integrability of $\beta(t)$ as presented in \cref{condition:sweeping-regular}, $x(t)$ lies in a compact ball and the result follows by \cite[p. 52, Corollary 2.15]{2012odeAndDS}.
\end{proof}

In the subsequent analysis, we assume that the solution to the PSP is defined on the entire nonnegative real line $\mathbb{R}^+$. We firstly consider the straightforward case where $\psi$ is \emph{strongly monotone}.

\begin{lemma}
  \label{lemma:global-stable-sweeping-monotonicity}
  Let the conditions of \cref{lemma:global-solution-sweeping-process} hold. Assume that $\psi$ satisfies the condition for strong monotonicity, i.e.,
  \begin{equation}
    \label{eqn:strong-monotonicity}
    \langle \psi(t,x) - \psi(t,y), x-y \rangle \ge \gamma(t) \|x-y\|^2,
  \end{equation}
  for all $t \in \mathbb{R}$, $x,y \in \mathcal{H}$ and a nonnegative continuous function $\gamma: \mathbb{R} \to \mathbb{R}^+$ satisfying $\int_{T}^{+\infty} \gamma(\tau) d\tau = +\infty$ for any fixed $T$. Then the solution to the sweeping process \cref{eqn:sweeping-process} is globally stable, i.e., $\|x(t)-y(t)\| \to 0$ for two trajectories with arbitrary initial values $x_0,y_0 \in C$ as $t \to +\infty$. 
\end{lemma}
\begin{proof}
  Let $x(t)$ and $y(t)$ be two solutions of \cref{eqn:sweeping-process}. Consider a domain $I=[s,t]$ such that both $x(t)$ and $y(t)$ are defined and the derivatives exist. By definition of the normal cone to a convex set, we have
  \begin{equation}
    \langle u-v, x-y\rangle \ge 0, \quad \forall u\in N_C(x), v\in N_C(y).
  \end{equation}
  By definition of a sweeping process, it follows that
  \begin{subequations}
    \begin{equation}
      -\dot{x}(t) - \psi(t,x(t)) \in N_C(x(t)),
    \end{equation}
    \begin{equation}
      -\dot{y}(t) - \psi(t,y(t)) \in N_C(y(t)).
    \end{equation}
  \end{subequations}
  Hence we have
  \begin{equation}
    -\langle \psi(t,x) - \psi(t,y), x-y \rangle \ge \langle \dot{x}(t) - \dot{y}(t), x - y \rangle.
  \end{equation}
  Using the strong monotonicity condition, we obtain
  \begin{equation}
    \langle \dot{x}(t) - \dot{y}(t), x(t) - y(t) \rangle \le -\gamma(t) \|x(t) - y(t)\|^2.
  \end{equation}
  This is equivalent to
  \begin{equation}
    \frac{d}{dt}\|x(t) - y(t)\|^2 \le -\gamma(t) \|x(t) - y(t)\|^2.
  \end{equation}
  Using Gronwall's inequality, we have
  \begin{equation}
    \|x(t) - y(t)\|^2 \le \exp\left( - \int_s^t \gamma(\tau) d\tau \right) \|x(s) - y(s)\|^2
  \end{equation}
  for all $t > s$. Consider two different initial points $x(T_0) = x_0$ and $y(T_0) = y_0$ for some fixed $T_0$ and $x_0,y_0 \in C$. Letting $t \to +\infty$, we obtain
  \begin{equation}
    \lim_{t \to \infty} \|x(t) - y(t)\|^2 = 0,
  \end{equation}
  which completes the proof.
\end{proof}

\begin{remark}
  \cref{lemma:global-stable-sweeping-monotonicity} reveals that the dynamical system associated with the sweeping process has a unique $\omega$-limit set independent of the choice of initial points under certain conditions. Furthermore, if the closed subset $C \subseteq \mathcal{H}$ is bounded, the positive semitrajectory of the associated dynamical system will be precompact, and hence the $\omega$-limit set will be internally chain transitive \cite{2001ChainTransitivity}. 
\end{remark}

\begin{theorem}
  \label{theorem:conv-strong-monotonicity}
  Let $(\mathcal{H}, \mathbb{R}^+, \pi)$ be the dynamical system associated with a unique global solution (under the conditions of \cref{lemma:global-solution-sweeping-process}) to \cref{eqn:sweeping-process}. Assume that $\psi$ is strongly monotone. If $(\mathcal{H}, \mathbb{R}^+, \pi)$ is a gradient-like dynamical system with a Lyapunov function $V: \mathcal{H} \to \mathbb{R}$, the $\omega$-limit set $\Omega(x)$ of any point $x \in C$ for any closed and convex subset $C \subseteq \mathcal{H}$ satisfies
  \begin{equation}
    V(\pi(t,y)) = V(y) , \quad \forall y \in \Omega(x), \quad \forall t \ge 0.
  \end{equation}
\end{theorem}
\begin{proof}
  Denote the non-wandering set $\mathcal{J}_x^+$ of $x \in \mathcal{H}$ by 
  \begin{equation*}
    \mathcal{J}_x^+ := \{y \in \mathcal{H} | \exists t_n \to \infty, x_n \to x, \text{ s.t. }  \pi(t_n,x_n) \to y\}.
  \end{equation*}
  We first show that if $x$ is contained in its own non-wandering set, i.e., $x \in \mathcal{J}_x^+$, then $V(\pi(t,x)) = V(x)$ for all $t\ge 0$. In fact, since $\mathcal{J}_x^+ \subseteq \mathcal{J}_{\pi(t,x)}^+$ for all $t\ge 0$ by definition, there exists $\tilde{x}_n \to \pi(t,x)$, $t_n \to \infty$, such that $\pi(t_n,\tilde{x}_n) \to x$. Hence it follows that 
  \begin{equation*}
    V(x) = \lim_{n\to \infty} V(\pi(t_n,\tilde{x}_n)) \le \lim_{n\to \infty} V(\tilde{x}_n) = V(\pi(t,x)), 
  \end{equation*}
  for all $t \ge 0$. Since $(\mathcal{H}, \mathbb{R}^+, \pi)$ is a gradient-like dynamical system, for all $t\ge 0$ and $x\in \mathcal{H}$ we have $V(\pi(t,x)) \le V(x)$. Therefore, it can be concluded that $V(\pi(t,x)) = V(x)$. It is also sufficient to observe that $\Omega(x) \subseteq \mathcal{J}_x^+$. We can conclude that if $x \in \Omega(x)$, then $V(\pi(t,x)) = V(x)$ for all $t\ge 0$.

  By \cref{lemma:global-stable-sweeping-monotonicity} and the remark following the lemma, $\Omega(x) = \Omega(y)$ for all $x,y \in C$. Therefore, the $\omega$-limit set can be denoted by $\Omega_C$. Since $C$ is closed in a complete space, it follows that $\Omega(x) \subseteq C$ for all $x \in C$. Hence we have $\Omega(x) = \Omega(v)$ for all $x \in C$ and $v \in \Omega_C$. For any $x \in C$ and any $u \in \Omega(x)$, we have $u \in \Omega(x) = \Omega_C = \Omega(u)$. Therefore, $V(u) = V(\pi(t,u))$ for all $t\ge 0$.
\end{proof}

To look closer at the fixed point set of the Lyapunov function $V$, it is sufficient to take derivatives with respect to time, i.e.,
\begin{equation}
  \frac{dV(\pi(t,x))}{dt} = \left\langle \nabla V(\pi(t,x)), \frac{d\pi(t,x)}{dt} \right\rangle = 0,
\end{equation}
for all $t \ge 0$. Letting $t=0$, we obtain
\begin{equation}
  \langle \nabla V(x), \dot{x}(0) \rangle = 0.
\end{equation}
If we consider $\psi(0,x) = \nabla V(x)$ in \cref{eqn:sweeping-process}, it follows \cite[p. 266, Proposition 2]{1984DifferentialInclusions} that
\begin{equation}
  \|\mathcal{P}_{\mathcal{T}_C(x)} [-\nabla V(x)]\|^2 = 0.
\end{equation}
Hence, $0 \in \nabla V(x) + N_C(x)$. This means that $x$ is a stationary point of the constrained optimization problem $\min_{y\in C} V(y)$.

In fact, the strong monotonicity of the time-dependent vector field $\psi(t,x)$ implicitly indicates some kind of convexity in $x$ of the time-varying vector field. To further investigate the general case where the vector field is non-convex, it is necessary to consider the case where strong monotonicity is not satisfied. In this case, the $\omega$-limit set is not unique compared to the conditions of \cref{lemma:global-stable-sweeping-monotonicity}, while it is still possible to generalize this result.

\begin{theorem}
  \label{theorem:compact-case-sweeping}
  Let $(\mathcal{H}, \mathbb{R}^+, \pi)$ be the dynamical system associated with a unique global solution (under the conditions of \cref{lemma:global-solution-sweeping-process}) to the sweeping process \cref{eqn:sweeping-process}. If $(\mathcal{H}, \mathbb{R}^+, \pi)$ is a gradient-like dynamical system with a Lyapunov function $V: \mathcal{H} \to \mathbb{R}$, the $\omega$-limit set $\Omega(x)$ of any point $x \in C$ for any closed, \textbf{bounded} and convex subset $C \subseteq \mathcal{H}$ satisfies
  \begin{equation}
    V(\pi(t,y)) = V(y) , \quad \forall y \in \Omega(x), \forall t \ge 0.
  \end{equation}
\end{theorem}

\begin{proof}
  For an arbitrary point $x \in C$, we can define a continuous function $\phi_x: \mathbb{R} \to \mathbb{R}, t \mapsto V(\pi(t,x))$. Clearly, we have $\phi_x(s) \le \phi_x(t)$ for all $s \ge t$. Since $C$ is compact, the positive semitrajectory of $\pi(t,x)$ is precompact and hence $V(\pi(t,x))$ is bounded. Hence $\phi_x(t)$ is a continuous bounded monotonically decreasing function of $t$. Therefore, there exists $\sigma_x \in \mathbb{R}$ such that $\lim_{t\to \infty} \phi_x(t) = \sigma_x$. Now consider $y \in \Omega(x)$. Then by definition, there exists $\tilde{t}_n \to \infty$ such that $\pi(\tilde{t}_n,x)\to y$. Consequently, $V(y) = \lim_{n\to \infty} V(\pi(\tilde{t}_n,x)) = \sigma_x$. This indicates that $\forall y \in \Omega(x)$, we have $V(y) = \sigma_x$. Since the $\omega$-limit set is invariant, we have $\pi(t,y) \in \Omega(x)$ for all $t \ge 0$. It then follows that $V(\pi(t,y)) = \sigma_x = V(y)$ for all $t\ge 0$.
\end{proof}

\cref{theorem:compact-case-sweeping} establishes a useful theoretical result on constrained continuous-time dynamical systems. Nevertheless, a discrete iteration is not guaranteed to remain stable under a general discretization scheme. Therefore, it is necessary to apply integrators which preserve certain structures of the continuous-time dynamics (especially the asymptotic behavior), as will be discussed in the next subsection. 

\subsection{Explicit Euler Scheme with Decaying Step Sizes}
Without loss of generality, we assume the conditions of \cref{lemma:global-solution-sweeping-process} are satisfied by the PSP \cref{eqn:sweeping-process}. Therefore, a unique solution is defined for the entire nonnegative real line $\mathbb{R}^+$ given any initial point.

To discretize the continuous-time process, we apply a time-decaying positive step size $h_k > 0$ ($\forall k \in \mathbb{N}^+$) which satisfies
\begin{equation}
  \label{eqn:step-size-requirement}
  h_0=0, \quad \lim_{k\to\infty} h_k=0, \quad \sum_{k=1}^\infty h_k = \infty.
\end{equation}
Correspondingly, the numerical scheme is given by
\begin{equation}
  \label{eqn:projected-numerical}
  \bar{z}_{k} = P_{\mathcal{K}} [\bar{z}_{k-1} - h_k \psi(t_{k-1}, \bar{z}_{k-1})], \quad \forall k \in \mathbb{N},
\end{equation}
where $\bar{z}_0 = z(0) = z_0 \in \mathcal{K}$ and $t_k=\sum_{\ell=0}^{k} h_\ell$. Recall that $x-\bar{x} \in N_{\mathcal{K}}(\bar{x})$ for all $x \in \mathcal{H}$ and $\bar{x} = P_{\mathcal{K}}(x)$. Hence the numerical scheme \cref{eqn:projected-numerical} can be viewed as
\begin{equation}
  \label{eqn:projected-numerical2}
  - \frac{\bar{z}_{k}-\bar{z}_{k-1}}{h_k} \in \psi(t_{k-1}, \bar{z}_{k-1}) + N_{\mathcal{K}}(\bar{z}_{k}), \quad \forall k \in \mathbb{N},
\end{equation}
which is a discrete explicit Euler scheme of the continuous-time dynamics \cref{eqn:sweeping-process} with step size $h_k > 0$ for all $k$. Furthermore, it is sufficient to consider a linear interplation process $u(t)$ for estimation, i.e., for all $k \in \mathbb{N}^+$
\begin{equation}
  u(t) = \bar{z}_{k-1} + \frac{\bar{z}_{k}-\bar{z}_{k-1}}{h_k} (t-t_{k-1}), \quad \forall t_{k-1}\le t < t_{k}.
\end{equation}
Let $z^s(t)$ represent the unique solution to the PSP \cref{eqn:sweeping-process} starting at $s$, i.e., for $z^s(s) = u(s)$
\begin{equation}
  -\dot{z}^s(t) \in \psi(t, z^s(t)) + N_{\mathcal{K}}(z^s(t)), \quad t \ge s.
\end{equation}
Likewise, denote by $z_s(t)$ the unique solution to the PSP \cref{eqn:sweeping-process} ending at $s$, i.e., for $z_s(s) = u(s)$
\begin{equation}
  -\dot{z}_s(t) \in \psi(t, z_s(t)) + N_{\mathcal{K}}(z_s(t)), \quad t \le s.
\end{equation}
To derive the convergence results, the following common assumption is introduced:
\begin{assumption}
  \label{assumption:conv-for-iterates}
  The following conditions hold:
  \begin{itemize}
    \item The sequence $\{z_n\}$ is bounded;
    \item The function $\psi$ is sweeping-regular (cf. \cref{condition:sweeping-regular}) on $([0,t],\mathcal{K})$, $\forall t \ge 0$; 
    \item The step size $\{h_k\}$ satisfies \cref{eqn:step-size-requirement} and $\sum_{k=1}^\infty h_k^3 < \infty$;
    \item $\psi$ satisfies the weak monotonicity condition:
    \begin{equation}
      \label{eqn:weak-monotonicity}
      \langle \psi(t,x) - \psi(t,y), x-y \rangle \ge \gamma(t) \|x-y\|^2,
    \end{equation}
    for all $t \in \mathbb{R}$, $x,y \in \mathcal{H}$ and an integrable function $\gamma: \mathbb{R} \to \mathbb{R}$ satisfying for all $T>0$
    \begin{equation}
      \inf_{\{t,s\in \mathbb{R}:0 \le t-s\le T\}} \int_{s}^{t} \gamma(\tau) d\tau > -\infty;
    \end{equation}
    \item Bounded variations: ($\forall M >0, \forall k \in \mathbb{N}^+$)
    \begin{equation}
      \sup_{\|z\|\le M} \|\psi(t_k,z) - \psi(t_{k-1},z)\| \le S_k = S_k(M),
    \end{equation}
    for $\{S_k,h_k\}$ satisfying $\sum_{k=0}^{\infty} S_k h_k (S_k + h_k) < \infty$.
  \end{itemize}
\end{assumption}
We then have the following lemma:
\begin{lemma}
  \label{lemma:convergence-limiting-behavior}
  Let \cref{assumption:conv-for-iterates} hold. For all $\tau >0$,
  \begin{subequations}
    \begin{equation}
      \lim_{s \to \infty} \sup_{s \le t \le s+\tau} \|u(t) - z^s(t)\| = 0,
    \end{equation}
    \begin{equation}
      \lim_{s \to \infty} \sup_{s-\tau \le t \le s} \|u(t) - z_s(t)\| = 0.
    \end{equation}
  \end{subequations}
\end{lemma}
\begin{proof}
  It is sufficient to prove the claim for $z^s(t)$ as arguments for the other claim are completely analogous. Furthermore, if the following alternative claim:
  \begin{equation}
    \lim_{\ell \to \infty} \sup_{t_\ell \le t \le t_\ell+\tau} \|u(t) - z^{t_\ell}(t)\| = 0, \quad \forall \tau > 0,
  \end{equation}
  holds, the other direction holds by analogy. Then, for all $s>0$, there exists some sufficiently large $\ell > 0$ such that $t_\ell \le s < s+\tau \le t_\ell + T_s$ for some $T_s > 0$ and 
  \begin{equation}
    \sup_{s \le t \le s+\tau} \|u(t) - z^s(t)\| \le \sup_{t_\ell \le t \le t_\ell + T_s} \|u(t) - z^{t_\ell}(t)\|. 
  \end{equation}
  The desired result will be obtained by taking limit.

  To begin with, we first show that $\dot{u}(t)$ is bounded in $[t_k, t_k+\tau]$ for all $\tau > 0$ and $k \in \mathbb{N}$. Without loss of generality, we assume that $N = N(\tau) =\sup \{m:t_m \le \tau\} \ge k+1$. Using the numerical scheme \cref{eqn:projected-numerical2}, we obtain
  \begin{equation}
    - v_k - \psi(t_k, \bar{z}_k) \in N_{\mathcal{K}}(\bar{z}_{k+1}), \quad \forall k \in \mathbb{N},
  \end{equation}
  where $v_k = (\bar{z}_{k+1}-\bar{z}_k)/h_{k+1}$. Applying the geometric characteristics of normal cones and making difference between $v_\ell$ and $v_{\ell-1}$, we find that
  \begin{equation}
    \langle v_\ell - v_{\ell - 1}, v_\ell \rangle \le - \langle \psi(t_\ell, \bar{z}_\ell) - \psi(t_{\ell-1}, \bar{z}_{\ell-1}), v_\ell \rangle.
  \end{equation}
  By \cref{assumption:conv-for-iterates}, it follows that
  \begin{equation}
    \|\psi(t_\ell, \bar{z}_\ell) - \psi(t_{\ell-1}, \bar{z}_\ell)\| \le S_\ell,
  \end{equation}
  for some $S_\ell>0$, and 
  \begin{equation}
    \|\psi(t_{\ell-1}, \bar{z}_\ell) - \psi(t_{\ell-1}, \bar{z}_{\ell-1})\| \le h_\ell L_\eta(t_{\ell-1}) \|v_{\ell-1}\|.
  \end{equation}
  Using the arithmetic mean inequality, i.e.,
  \begin{equation}
    ab \le \frac{1}{2}(b^2 c + a^2/c), \quad \forall a,b\in \mathbb{R}, c > 0,
  \end{equation}
  and taking some $0<\varepsilon<1$, we conclude that 
  \begin{equation}
    (1-\varepsilon) \|v_\ell\|^2 \le \frac{1+(L_\eta(t_{\ell-1}))^2 h_\ell^2}{2\varepsilon} \|v_{\ell-1}\|^2 + \frac{S_{\ell}^2}{\varepsilon}.
  \end{equation}
  Letting $\varepsilon = 1/2$, we have for all $\ell \ge 0$
  \begin{equation}
    \|v_\ell\|^2 \le 2[1+(L_\eta(t_{\ell-1}))^2 \bar{h}^2] \|v_{\ell-1}\|^2 + 2 S_{\ell}^2,
  \end{equation}
  where $\bar{h} = \sup_{k\in \mathbb{N}} h_k$ is an upper bound of $\{h_k\}$. By the discrete Gronwall inequality, it follows that for all $\ell > k$
  \begin{equation*}
    \|v_\ell\|^2 \le (2+2L^2\bar{h}^2)^{\ell-k} \|v_k\|^2 + 2S_{\ell}^2 \sum_{m=k}^\ell (2+2L^2\bar{h}^2)^{\ell-m},
  \end{equation*}
  for some $L>0$ due to the integrability of $L_\eta(t)$. Therefore, we conclude that for any fixed $\tau > 0$, $v_\ell$ is bounded for all $\ell \le N$ and $\dot{u}(t)$ is bounded as a direct result. 

  Next we estimate $\|u(t) - z^{t_\ell}(t)\|$. Let $t_\ell \le t < t_{\ell+1}$. It is clear that we have the following truncated dynamics:
  \begin{subequations}
    \begin{equation}
      -\dot{u}(t)-\psi(t_\ell, \bar{z}_\ell) \in N_{\mathcal{K}}(\bar{z}_{\ell+1}),
    \end{equation}
    \begin{equation}
      -\dot{z}^{t_\ell}(t)-\psi(t,z^{t_\ell}(t)) \in N_{\mathcal{K}}(z^{t_\ell}(t)).
    \end{equation}
  \end{subequations}
  Applying the geometric characteristics of normal cones, it is straightforward to conclude that
  \begin{equation}
    \label{eqn:temp-res}
    \frac{1}{2} \frac{d}{dt}\|u(t) - z^{t_\ell}(t)\|^2 \! \le \! - \langle \psi(t_\ell, \bar{z}_{\ell})-\psi(t,z^{t_\ell}(t)), u(t) - z^{t_\ell}(t)\rangle.
  \end{equation}
  It follows from the boundedness of $\dot{u}(t)$ that
  \begin{equation*}
    \|u(t) - \bar{z}_\ell\| = \left\|\frac{\bar{z}_{\ell+1}-\bar{z}_{\ell}}{h_{\ell+1}} (t-t_\ell) \right\| \le \|\dot{u}(t)\|h_{\ell+1} \le M h_{\ell+1}.
  \end{equation*}
  Therefore, we have for some $K>0$ by \cref{assumption:conv-for-iterates}
  \begin{equation*}
    \begin{aligned}
      &\|\psi(t,u(t)) - \psi(t_\ell, \bar{z}_\ell)\| \\
      &\le \|\psi(t,u(t)) - \psi(t_\ell,u(t)) \| + \| \psi(t_\ell,u(t)) - \psi(t_\ell, \bar{z}_\ell)\| \\ 
      &= S_\ell + K h_{\ell+1}.
    \end{aligned}
  \end{equation*}
  Recall that $\psi$ satisfies the weak monotonicity condition \cref{eqn:weak-monotonicity}. Estimating the right-hand side of \cref{eqn:temp-res}, we conclude that
  \begin{equation}
    \begin{aligned}
      &- \langle \psi(t_\ell, \bar{z}_{\ell})-\psi(t,z^{t_\ell}(t)), u(t) - z^{t_\ell}(t)\rangle \\ &\le \frac{(S_\ell+K h_{\ell+1})^2}{2} - \frac{1}{2}(\gamma(t)-1) \|u(t) - z^{t_\ell}(t)\|^2,
    \end{aligned}
  \end{equation}
  Then \cref{eqn:temp-res} can be written as
  \begin{equation}
    \begin{aligned}
      &\frac{d}{dt}\|u(t) - z^{t_\ell}(t)\|^2 \\
      &\le -(\gamma(t)-1) \|u(t) - z^{t_\ell}(t)\|^2 + (S_\ell+K h_{\ell+1})^2.
    \end{aligned}
  \end{equation}
  Now consider $t_\ell \le t \le t_\ell + \tau$. Using Gronwall's inequality and $u(t_{\ell}) = z^{t_\ell}(t_\ell)$, we obtain
  \begin{equation*}
    \|u(t) - z^{t_\ell}(t)\|^2 \! \le \! \sum_{k=\ell}^{N(\tau)-1} \! (S_k+K h_{k+1})^2 \! \int_{t_k}^{t_{k+1}} \! A(s,t) ds,
  \end{equation*}
  where $A(s,t) = e^{\int_s^t -(\gamma(\tau)-1) d\tau}$. By the weak monotonicity condition, we have for all $t_\ell \le s < t \le t_\ell + \tau$
  \begin{equation}
    \sup_{s,t} A(s,t) = \sup_{s,t} e^\tau e^{\int_s^t -\gamma(x) dx} < \infty.
  \end{equation}
  Therefore, it follows that for some $M>0$
  \begin{equation}
    \|u(t) - z^{t_\ell}(t)\|^2 \le M \sum_{k=\ell}^{N(\tau)-1} \! (S_k+K h_{k+1})^2 h_{k+1}.
  \end{equation}
  By bounded variations in \cref{assumption:conv-for-iterates}, it follows that
  \begin{equation}
    \lim_{\ell\to\infty} \sup_{t_\ell \le t \le t_\ell + \tau} \|u(t) - z^{t_\ell}(t)\|^2 = 0,
  \end{equation}
  which completes the proof.
\end{proof}

% When the primary function $\psi$ is independent of $t$, \cref{assumption:conv-for-iterates} can be relaxed. In fact, we have instead the following lemma and a corresponding weaker assumption.

% \begin{lemma}
%   Let $\psi$ be independent of $t$, i.e., $\psi(t,x)=\psi(x)$. Assume that $\{z_n\}$ is bounded, $\psi$ is Lipschitz-continuous, and the step size sequence $\{h_k > 0\}$ satisfies $\sum_{k=1}^\infty h_k = \infty$. The conclusion of \cref{lemma:convergence-limiting-behavior} still holds.
% \end{lemma}

% \begin{proof}
  
% \end{proof}

Based on this lemma, via a straightforward application of \cite[p. 17, Theorem 2.1]{2023StochasticApproxDynamics}, we obtain the desired convergence result as follows:
\begin{theorem}
  \label{theorem:conv-numerical}
  Under \cref{assumption:conv-for-iterates}, the sequence $\{\bar{z}_n\}$ generated by \cref{eqn:projected-numerical} converges to a connected internally chain transitive invariant set of \cref{eqn:sweeping-process}.
\end{theorem}

In general, \cref{theorem:conv-numerical} is the best result one can obtain on convergence of the numerical scheme \cref{eqn:projected-numerical2} corresponding to a PSP. Unfortunately, the result presented in \cref{theorem:compact-case-sweeping} for a continuous-time sweeping process cannot be simply extended to the numerical case. The primary obstacle lies in the unboundedness of the trajectory as discussed in \cite{1996DynSystApprochStocApprox}. Besides, some alternatives for the assumption that $\{\bar{z}_k\}$ is bounded are provided in \cite[Chap. 4]{2023StochasticApproxDynamics}. Furthermore, the following corollary is immediate.

\begin{corollary}
  If the only internally chain transitive invariant sets for \cref{eqn:sweeping-process} are isolated critical points, then $\{z_n\}$ converges to a critical point under \cref{assumption:conv-for-iterates}.
\end{corollary}

In the previous subsection, we have characterized the $\omega$-limit set of the continuous-time sweeping process on a compact and convex subset given the existence of a Lyapunov function. The question is whether this result can be extended to the numerical case. Such extensions are never straightforward since the $\omega$-limit set of $u(t)$ only coincides with an internally chain transitive set of \cref{eqn:sweeping-process} as presented in \cref{theorem:conv-numerical}. Although the $\omega$-limit set of any precompact positive orbit with respect to a continuous semiflow is internally chain transitive \cite[Lemma 2.1']{2001ChainTransitivity}, the opposite is not true in general. Fortunately, by introducing the concept of Lyapunov functions for a PSP, we can obtain a similar conclusion to that of the continuous-time case.
\begin{corollary}
  \label{corollary:projected-numerical-Lyapunov-stability}
  Let $\mathcal{L} \subset \mathbb{R}^m$ be a nonempty compact set, $U \subset \mathcal{K} \subset \mathbb{R}^m$ be a bounded open neighborhood of $\mathcal{L}$, and $V:\mathcal{K} \to \mathbb{R}^+$ be continuously differentiable. Let the following hold:
  \begin{itemize}
    \item $u(t) \in U$ for all $t\ge 0$;
    \item $V^{-1}(0) = \mathcal{L}$;
    \item The Lie derivative $\frac{dV}{dt} \le 0$ along \cref{eqn:sweeping-process} holds for all $t\ge 0$ and $x \in \mathcal{K}$ with equality if and only if $x \in \mathcal{L}$.
  \end{itemize}
  Then $\{z_n\}$ converges to an internally chain transitive set contained in $\mathcal{L}$ under \cref{assumption:conv-for-iterates}. 
\end{corollary}
\begin{proof}
  Note that the corollary is inspired by \cite[p. 19, Corollary 2.1]{2023StochasticApproxDynamics}; we reproduce the proof for the sake of completeness. Let $M = \sup_n \|z_n\| < \infty$ and $C = \sup_{\|z\|\le M} V(z)$. For any constant $0 < b \le C$, we define $Z^b := \{x \in U: V(x) < b\}$. For $0< \epsilon < C/2$, we have
  \begin{equation}
    -\zeta := \sup_{t \ge 0,x \in \bar{Z}^C \backslash Z^\epsilon} \frac{dV}{dt}(t,x) < 0,
  \end{equation}
  where $\bar{Z}^C$ denotes the clousre of $Z^C$. It then follows that
  \begin{equation*}
    V(z(t)) = V(z(0)) + \int_0^t \frac{dV}{ds}(s,z(s)) ds \le V(z(0)) - t \zeta.
  \end{equation*}
  Let $\tau$ be an upper bound on the time required for a solution to \cref{eqn:sweeping-process} starting from $\bar{Z}^C$ to reach $Z^\epsilon$. Hence, we can pick $C/\zeta < \tau < \infty$. Since $\mathcal{K}$ is compact and $V$ is continuously differentiable, $V$ is Lipschitz continuous in $\mathcal{K}$. Then there exists some $\delta > 0$ such that for all $x \in \bar{Z}^C$ and $y \in \mathcal{K}$ with $\|x-y\| < \delta$, we have $|V(x)-V(y)| < \epsilon$. By \cref{lemma:convergence-limiting-behavior}, there exists $t_0$ such that for all $t\ge t_0$, we have $\sup_{t\le s \le t+\tau}\|u(s)-z^t(s)\| < \delta$. Since $u(s) \in \bar{Z}^C$, it follows that $|V(u(t+\tau))-V(z^t(t+\tau))| < \epsilon$, and hence $u(t+\tau) \in Z^{2\epsilon}$ for $z^t(t+\tau) \in Z^\epsilon$. Therefore, $u(t) \in Z^{2\epsilon}$ for all $t\ge t_0+\tau$. Letting $\epsilon \downarrow 0$, we have $u(t) \to \mathcal{L}$ as $t\to \infty$.
\end{proof}

Although this corollary provides a useful tool to address optimization problems in smooth analysis, it fails to apply to composite optimization problem where $V$ is nonsmooth. By contrast, the subsequent theorem offers a general framework.
\begin{theorem}
  \label{theorem:general-func-V}
  Let $\Lambda \subset \mathbb{R}^m$ be any subset. Suppose that $V: \mathbb{R}^m \to \mathbb{R}$ is a Lyapunov function for $\Lambda$ with respect to the trajectory of \cref{eqn:sweeping-process}. Assume that $V(\Lambda)$ has empty interior. Then $\{z_n\}$ converges to an internally chain transitive set $\mathcal{L}$ contained in $\Lambda$ under \cref{assumption:conv-for-iterates} and $V$ is constant in $\mathcal{L}$. 
\end{theorem}
\begin{proof}
  The results follow from \cite[Proposition 3.27]{2005StochasticApproxDifferentialInclusion} and \cref{theorem:conv-numerical}.
\end{proof}

% In terms of the hybrid sweeping process \cref{eqn:hybrid-sweeping-process}, it is sufficient to let $\mathcal{K} = \mathcal{H}_1 \times \mathscr{C}$ for a compact and convex subset $\mathscr{C} \subset \mathcal{H}_2$. In this case, by \cref{lemma:existence-solution-sweeping} it follows that for all $t \in [T_s,T_e]$ where a solution is defined
% \begin{equation}
%   \begin{aligned}
%     \|p(t)\| &\le \|p(T_s)\| + \int_{T_s}^{T_e} \|\phi(s,p(s),q(s))\| ds \\
%     &\le \|p(T_s)\| + D\int_{T_s}^{T_e} \beta(\tau) d\tau < \infty,
%   \end{aligned}
% \end{equation}
% which indicates that the momentum $p(t)$ is uniformly bounded, and hence $z(t)=(p(t),q(t))$ lies in a fixed compact ball for all $t \in [T_s,T_e]$. Therefore, the first item in \cref{assumption:conv-for-iterates} can be replaced by compactness of $\mathscr{C}$.

Indeed, the conclusions regarding the numerical scheme, which are derived from an initial continuous-time PSP, can be viewed from the reversed direction. Specifically, given a numerical scheme
\begin{equation}
  z_{k+1} = P_{\mathcal{K}} [z_k + h_{k+1} \phi_k(z_k)],
\end{equation}
along with its corresponding continuous-time dynamics
\begin{equation}
  -\dot{z}(t) \in \psi(t,z) + N_{\mathcal{K}}(z),
\end{equation}
where $\psi(t_k,z) = \phi_k(z)$ for all $k\in\mathbb{N}$ and $z \in \mathcal{H}$, it follows that the aforementioned conclusions still hold.

\begin{remark}
  Consider a perturbed stochastic scheme, i.e.,
  \begin{equation}
    \tilde{z}_{k+1} = P_{\mathcal{K}}[\tilde{z}_k - h_{k+1}\psi(t_k,\tilde{z}_k) - h_{k+1}(U_{k+1}+r_{k+1})],
  \end{equation} 
  where $\{U_k\}$ and $\{r_k\}$ are sequences of random perturbations. Assume that for each $T>0$, 
  \begin{equation}
    \label{eqn:small-perturbation}
    \lim_{n \to \infty} \sup_{\{k:0\le t_k - t_n \le T\}} \left\| \sum_{\ell=n}^{k} h_{\ell} U_\ell \right\| = 0, \ \text{a.s.},
  \end{equation}
  and $\lim_{k \to \infty} r_k = 0$ a.s. Then the conclusions a,,nd corresponding analysis above for a deterministic scheme hold almost surely for the stochastic scheme, following standard analysis of the classical result \cite{1996DynSystApprochStocApprox}. 
\end{remark}

Via straightforward application of this remark, we immediately obtain the following useful corollary:
\begin{corollary}
  \label{corollary:asymptotic-general-conv}
  Consider an asymptotic numerical scheme
  \begin{equation}
    \label{eqn:asymptotic-numerical-scheme}
    \bar{z}_{k+1} = P_{\mathcal{K}}[\phi_k(\bar{z}_k) - h_{k+1} (\psi(t_k,\bar{z}_k)+\xi_{k+1})],
  \end{equation}
  where $\phi_k: \mathbb{R}^m \to \mathbb{R}^m$ is continuous for all $k$, and $\{\xi_k\}$ is a sequence of random perturbations satisfying \cref{eqn:small-perturbation}. Let \cref{assumption:conv-for-iterates} hold. Assume that
  \begin{equation}
    \lim_{n \to \infty} \sup_{x \in \mathbb{R}^m} \left\| \phi_n (x) - x \right\|/h_{n+1} = 0.
  \end{equation}
  The conclusion of \cref{theorem:general-func-V} holds for \cref{eqn:asymptotic-numerical-scheme}.
\end{corollary}

\begin{figure}[t]
  \centering
  \includegraphics[width=0.48\textwidth]{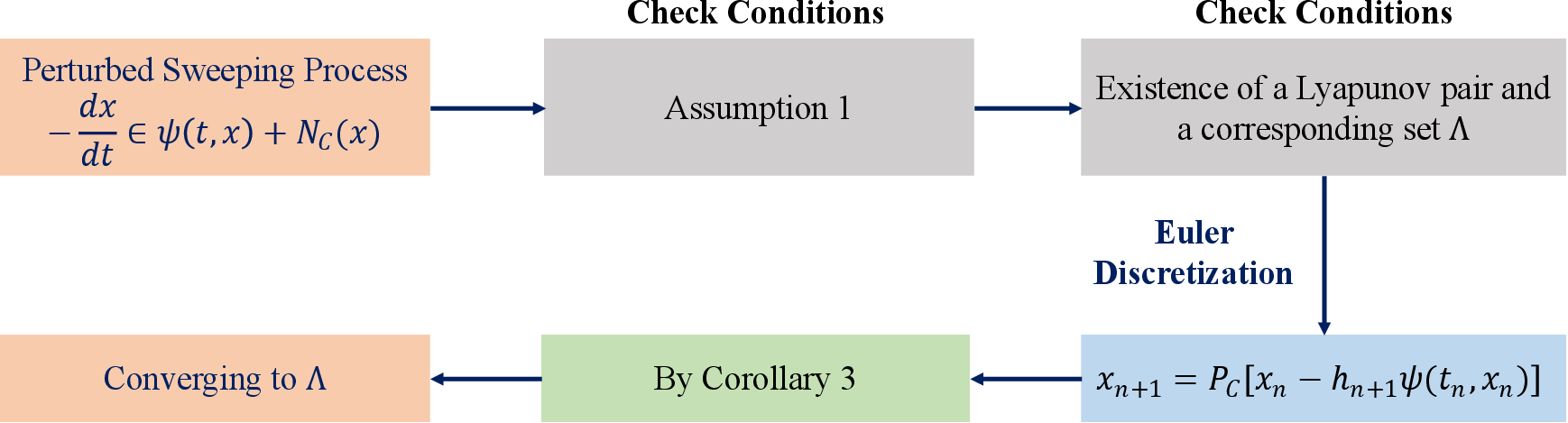}  
  \caption{The conceptual framework to design a convergent projected method.}
  \label{fig:conv-plot}
\end{figure}

\begin{remark}
  A typical example of $\{\xi_k\}$ which satisfies \cref{eqn:small-perturbation} is martingale difference noise. Use $\mathscr{F}_k$ to represent the filtration of $\{\bar{z}_1,\bar{z}_2,\dots,\bar{z}_k\}$. Let $\xi_k$ satisfy $\mathbb{E}[\xi_{k+1}|\mathscr{F}_k] = 0$ and $\mathbb{E}[\|\xi_{k+1}\|^2|\mathscr{F}_k] \le \mu(1+\|\bar{z}_k\|^2)$ for some constant $\mu > 0$ for all $k$. We consider the finite sum $\zeta_n = \sum_{i=1}^n h_i \xi_i$. Since $\sum_{i=1}^\infty h_i^2 < \infty$ and $\{\bar{z}_k\}$ is bounded a.s., we have
  \begin{equation}
    \sum_{\ell = 0}^\infty \mathbb{E}[\|\zeta_{\ell+1} - \zeta_\ell\|^2|\mathscr{F}_\ell] = \sum_{\ell = 0}^\infty h_\ell^2 \mathbb{E}[\|\xi_{\ell+1}\|^2|\mathscr{F}_\ell] < \infty
  \end{equation}
  a.s. It follows by Doob's martingale convergence theorem that $\{\zeta_n\}$ converges. Therefore, we conclude that
  \begin{equation}
    \lim_{n\to \infty} \sup_{m>0} \|\zeta_{n+m} - \zeta_n\| = 0, \text{ a.s.}
  \end{equation}
\end{remark}

With the above theoretical results, we summarize a conceptual framework to design a provably convergent projected compressed method as presented in \cref{fig:conv-plot}. We note that it is quite tricky to numerically analyze a momentum-based method for non-convex optimization problems especially with a biased compressor. In addition, such analysis is usually case-by-case due to lack of deep understanding of underlying dynamical representations. By contrast, this unifying framework offers a convenient way to guarantee theoretical convergence.

%% file: sections/application.tex
\vspace{-8pt}
\section{Application to Constrained Optimization}
\label{sec:appl}

In this section, we firstly justify the validity of the above established theoretical results by providing some examples of convergence analysis of projected variants of existing popular optimization methods. Moreover, we present projected compressed schemes with compressors, of which the convergence can be immediately established by \cref{corollary:asymptotic-general-conv}.

\vspace{-6pt}
\subsection{Schemes with Exact Inputs}
As demonstrated in \cref{assumption:conv-for-iterates}, the time-varying vector field is not restricted to be continuous with respect to time (the first argument). Therefore, it is feasible to add countable bounded jump discontinuities to the vector field. This, in turn, supports numerical schemes with a constant step size as we can add a cofactor to cancel the vanishing step size. 

We begin with the standard stochastic gradient descent for constrained optimization.

\textbf{\emph{Example 1}: Stochastic projected gradient descent (PGD).} Consider the dynamics given by
\begin{equation}
  \frac{dz}{dt} \in -\nabla f(z) - N_{\mathscr{C}}(z(t)),
\end{equation}
where $\mathscr{C} \subset \mathbb{R}^m$ is a compact convex set, $f:\mathbb{R}^m \to \mathbb{R}$ is lower-bounded and has a Lipschitz-continuous gradient. Taking the step size $h_k = 1/k$, we have for a martingale difference noise sequence $\{\xi_k\}$ with bounded variance:
\begin{equation}
  \label{eqn:stochastic-pgd}
  z_{n+1} = P_{\mathscr{C}}[z_n - h_{n+1} (\nabla f(z_n) + \xi_{n+1})],
\end{equation}
which is the classical form of a stochastic projected gradient descent method. Its convergence can be established immediately by selecting $f$ as the Lyapunov function and using the facts from \cite[Proposition 2]{2006Equivalence} according to \cref{theorem:general-func-V}.

\textbf{\emph{Example 2}: Projected Nesterov accelerated gradient (PNAG).} Consider the following perturbed sweeping process:
\begin{equation}
  \label{eqn:proj-continuous-Nesterov-form2}
  \left\{
  \begin{aligned}
    \frac{dx}{dt} &= \kappa(t) [y -x - \gamma \nabla f(y)], \\
    \frac{dy}{dt} &\in \kappa(t) [\mu y - \mu x- \nu \nabla f(y) - N_{C}(y)],
  \end{aligned}
  \right.
\end{equation}
where $\kappa(t)$ is defined by $\kappa(t):= \sup\{\sqrt{k+1}: k\in \mathbb{N}, \tau_k\le t\}$, $\tau_k = \sum_{\ell=0}^k \sqrt{\ell}$ for all $k \in \mathbb{N}$, and $\gamma>0$ and $0<\mu<1$ are positive constants. $f:\mathbb{R}^m \to \mathbb{R}$ is lower-bounded and has a Lipschitz-continuous gradient, $\nu = \gamma (1+\mu)$, and $C \subset \mathbb{R}^m$ is a compact and convex subset. Taking the step size $h_{k+1} = 1/\kappa(t_k)$, the Euler discretization produces 
\begin{equation}
  \label{eqn:PNAG-form2}
  \left\{
  \begin{aligned}
    x_{n+1} &= y_n - \gamma \nabla f(y_n),\\
    y_{n+1} &= P_C[x_{n+1} + \mu(x_{n+1} - x_n)].
  \end{aligned}
  \right.
\end{equation}
Then we have the following corollary:
\begin{corollary}
  \label{corollary:conv-proj-Nesterov-form2}
  Let $\mathcal{L}$ be the set of KKT points of $f$ on $C$ associated with the constrained optimization problem $\min_{x \in C} f(x)$. If $f(\mathcal{L})$ has empty interior, $(y_n)$ of the numerical scheme \cref{eqn:PNAG-form2} converges to $\mathcal{L}$.
\end{corollary}
\begin{proof}
  Denote the composite vector field $\psi(t,p,q)$ by
  \begin{equation}
    \psi(t,x,y) := \left(
    \begin{matrix}
      \kappa(t) [y -x - \gamma \nabla f(y)] \\
      \kappa(t) [\mu y - \mu x- \nu \nabla f(y)]
    \end{matrix}
    \right).
  \end{equation}
  Consider the differentiable function $V:\mathbb{R}^m\times \mathbb{R}^m \to \mathbb{R}$ as
  \begin{equation}
    V(x,y) = \frac{1}{2}\|x-y\|^2 + \gamma f(y).
  \end{equation}
  Take $W(t,x,y) = \kappa(t) ((1-\mu)\|x-y\|^2 + R(y))$ for
  \begin{equation}
    \begin{aligned}
      R(y) &= \gamma \nu \langle -\nabla f(y) - u_0, -\nabla f(y) \rangle \\
      &= \gamma \nu \|P_{\mathcal{T}_C(y)}[-\nabla f(y)]\|^2,
    \end{aligned}
  \end{equation}
  where $u_0 = P_{N_C(y)}[-\nabla f(y)]$. For all $u \in N_C(y)$, 
  \begin{equation}
    \begin{aligned}
      \langle \nabla V(x,y), -\psi(t,x,y)-(0,u)^T \rangle = -W(t,x,y) \\ - \gamma \mu \|u_0\|^2 - \langle y-x+\gamma \nabla f(y), u \rangle.
    \end{aligned}
  \end{equation}
  Since $0 \in N_C(y)$, we conclude that for $\zeta = \nabla V(x,y)$
  \begin{equation}
    \min_{\|u\| \le \psi(t,x,y)} \langle \zeta, -\psi(t,x,y)-(0,u)^T \rangle + W(t,x,y) \le 0.
  \end{equation}
  By \cref{lemma:Lyapunov-pair}, $(V,W)$ is a Lyapunov pair. Observe that only if $W=0$, $V(x(t),y(t)) \le V(x(s),y(s))$ for all $t > s$. Therefore, it is sufficient to consider the set $\Lambda = \{(x,y): x=y, P_{\mathcal{T}_C(y)}[-\nabla f(y)] = 0, y \in C\}$ such that $W(\Lambda) = 0$. It is clear that $\Lambda|_x$ coincides with the KKT point set of the constrained problem $\min_{x\in C} f(x)$. By assumption, $\Lambda$ is nonempty and $V(\Lambda)$ has empty interior. Clearly, $V$ is a Lyapunov function for $\Lambda$ by definition and the desired result can be obtained via \cref{theorem:general-func-V}. 
\end{proof}

\begin{remark}
  In practice, NAG is often utilized with time-varying step sizes, implying that the parameter $\mu_k$ evolves throughout the iterative process. As a consequence, the corresponding continuous-time model must be formulated to account for this temporal variability. Notably, the analysis presented herein remains valid in this context, given that the Lyapunov function $V$ is not explicitly dependent on time, i.e., $\partial V/\partial t = 0$. However, this is not generally the case as will be discussed in the next example.
\end{remark}

\textbf{\emph{Example 3}: Projected optimized gradient (POGM).} Consider the following dynamical system:
\begin{equation}
  \left\{
  \begin{aligned}
    \frac{dx}{dt} &= \kappa(t) [y -x - \gamma \nabla f(y)], \\
    \frac{dy}{dt} &\in \kappa(t) [ \mu(t) y - \mu(t) x- \beta(t) \nabla f(y) - N_{C}(y)],
  \end{aligned}
  \right.
\end{equation}
where $\beta(t) = \gamma (1+\mu(t) + \lambda(t))$ and the time-varying step sizes are defined as for all $n \in \mathbb{N}$ and $t_n \le t < t_{n+1}$
\begin{equation}
  1 > \mu(t) = \mu(t_n) = \mu_n > 0, \ \lambda(t) = \lambda(t_n) = \lambda_n > 0,
\end{equation}
where $t_n = \sum_{\ell=0}^n h_\ell$ and $\sup_n \lambda_n < \infty$. The other parameters are set according to Example 2. By Euler discretization, we obtain the following projected variant of the optimized gradient method \cite{2016OptimizedGradientDescent}:
\begin{equation}
  \label{eqn:POGM}
  \begin{aligned}
    x_{n+1} \!&= \! y_n - \gamma \nabla f(y_n), \\
    y_{n+1} \! &=\! P_C\! \left[ x_{n+1} \!+\! \mu_{n+1}\! (x_{n+1} \!- \!x_n) \!- \!\gamma \lambda_{n+1} \nabla f(y_n)\right],
  \end{aligned}
\end{equation}
Letting $\lambda_k / \mu_k$ decrease with respect to $k$, we have the following convergence result:
\begin{corollary}
  \label{corollary:conv-POGM}
  Let $\mathcal{L}$ be the set of KKT points of $f$ on $C$ associated with the constrained optimization problem $\min_{x \in C} f(x)$. If $f(\mathcal{L})$ has empty interior, $(y_n)$ of the numerical scheme \cref{eqn:POGM} converges to $\mathcal{L}$.
\end{corollary}
\begin{proof}
  Denote the composite vector field $\psi(t,p,q)$ by
  \begin{equation}
    \psi(t,x,y) := \left(
    \begin{matrix}
      \kappa(t) [y -x - \gamma \nabla f(y)] \\
      \kappa(t) [ \mu(t) y - \mu(t) x- \beta(t) \nabla f(y)]
    \end{matrix}
    \right).
  \end{equation}
  Consider the function $V: \mathbb{R} \times \mathbb{R}^m \times \mathbb{R}^m \to \mathbb{R}$ as
  \begin{equation}
    V(t,x,y) = \frac{1}{2}\|x-y\|^2 + \gamma \left( 1+ \frac{\lambda(t)}{\mu(t)}\right) f(y).
  \end{equation}
  Recall that $V$ is piecewisely explicitly independent of time in the interval $[t_k, t_{k+1})$ for all $k \ge 0$. Since $\lambda_k / \mu_k$ decreases with respect to $k$, it is sufficient to discuss the variation of $V$ piecewisely. For all $t \in [t_k, t_{k+1})$, we have
  \begin{equation}
    V(t,x,y) = V_k(x,y) = \frac{1}{2}\|x-y\|^2 + \sigma_k f(y),
  \end{equation}
  where $\sigma_k := \gamma ( 1+ \lambda_k/ \mu_k)$. Take $W(t,x,y) = \kappa(t) ((1-\mu(t))\|x-y\|^2 + R(y))$ for
  \begin{equation}
    \begin{aligned}
      R(y) &= \sigma_k \beta_k \langle -\nabla f(y) -w_0, -\nabla f(y) \rangle \\
      &= \sigma_k \beta_k \|P_{\mathcal{T}_C(y)}[-\nabla f(y)]\|^2,
    \end{aligned}
  \end{equation}
  where $w_0 = P_{N_C(y)}[-\nabla f(y)]$. For all $w \in N_C(y)$, 
  \begin{equation}
    \begin{aligned}
      \langle \nabla V_k(x,y), -\psi(t,x,y)-(0,u)^T \rangle = -W(t,x,y) \\ - \sigma_k \beta_k \|w_0\|^2 - \langle y-x+\sigma_k \nabla f(y), w \rangle.
    \end{aligned}
  \end{equation}
  Since $0 \in N_C(y)$, we conclude that for $\zeta = \nabla V_k(x,y)$
  \begin{equation}
    \min_{\|w\| \le \psi(t,x,y)} \langle \zeta, -\psi(t,x,y)-(0,w)^T \rangle + W(t,x,y) \le 0.
  \end{equation}
  The rest of the proof follows.
\end{proof}

\textbf{\emph{Example 4}: Decentralized constrained optimization.} If asymptotic consensus can be achieved, an unconstrained decentralized optimization scheme can be viewed as a combination of a centralized vector (the consensus vector) and a vanishing random perturbation (the consensus error or the random shuffling). Therefore, such a method is intentionally a variant of a centralized stochastic approximation scheme. 

The primary obstacle in analyzing constrained decentralized optimization methods results from the nonsmooth part of the iteration of the consensus vector. Consequently, it is significant to associate the iteration with a projected stochastic approximation scheme, of which the convergence can be established with the theoretical results in the previous section.

We take the classical multiagent projected method presented in \cite{2013ConvAMultiagentProjected} for example. Under certain assumptions on time-varying weighted graphs, the multiagent projected method can be transformed into the following projected stochastic approximation scheme in terms of the consensus vector $\theta_k$:
\begin{equation}
  \theta_{k+1} = P_{\mathscr{C}}[\theta_{k} - \gamma_{k+1} \nabla f(\theta_k) + \gamma_{k+1}(\xi_{k+1} + r_{k+1})],
\end{equation}
where $\{\xi_k\}$ and $\{r_k\}$ are random perturbations satisfying
\begin{equation}
  \lim_{k\to \infty} \sup_{\ell \ge k} \left| \sum_{n=k}^\ell \gamma_n \xi_n \right| = 0, \ \lim_{k\to \infty} r_k = 0, \ \text{a.s.},
\end{equation}
and $\{\gamma_k\}$ is the positive time-varying step size such that $\sum_{k} \gamma_k = \infty$ and $\sum_{k} \gamma^2_k < \infty$. Let $\mathscr{C}$ be a nonempty convex and compact subset in $\mathbb{R}^m$ and $\mathcal{L}$ be the KKT point set of $f$ on $\mathscr{C}$. Assume that $f(\mathcal{L})$ has empty interior. The convergence of this scheme can be immediately established by \cref{corollary:asymptotic-general-conv}.

\vspace{-6pt}

\subsection{Schemes with Compressors}
For convenience, we consider the usual Euclidean space $\mathbb{R}^m$ in this subsection. By a compressor, we mean a (probably stochastic) mapping $\vartheta$ such that 
\begin{equation}
  \label{eqn:unbiased-compressor}
  \mathbb{E}[\vartheta(x) - x] = 0, \ \mathbb{E}\|\vartheta(x) - x\|^2 \le \mu (1+\|x\|^2), 
\end{equation}
for some constant $\mu > 0$. We consider the following scheme:
\begin{equation}
  \label{eqn:proj-compress-sgd}
  y_{n+1} = P_{\mathscr{C}}[y_n - h_{n+1}\vartheta(\nabla f(y_n))],
\end{equation}
where the parameter settings are the same as Example 1 in the last subsection. Since $\vartheta$ is unbiased, we have
\begin{equation}
  y_{n+1} = P_{\mathscr{C}}[y_n - h_{n+1}(\nabla f(y_n) + \xi_n)],
\end{equation}
where $\xi_n$ represents the compression error, and it is direct to check that $\{\xi_n\}$ is a martingale difference noise sequence with 
\begin{equation*}
  \mathbb{E}[\|\xi_{n+1}\|^2|\mathscr{F}_n] \! \le \! \mu(1+\| \nabla f(y_{n+1}) \|^2) \! \le \! K(1+\|y_{n+1}\|^2),
\end{equation*}
a.s. due to the Lipschitz continuity of $\nabla f$, where $K>0$ is a constant and $\mathscr{F}_n := \sigma (x_\ell, \ell \le n)$ is the filtration generated by past parameters. Therefore, $(h_n,\xi_n)$ satisfies \cref{eqn:small-perturbation} by Doob's martingale convergence theorem. By \cref{corollary:asymptotic-general-conv}, \cref{eqn:proj-compress-sgd} converges to the KKT point set of $f$ on $\mathscr{C}$.

Note that similar analysis naturally applies to PNAG and POGM just by replacing the gradient $\nabla f(x)$ with $\vartheta(\nabla f(x))$.

\subsection{A Distributed Scheme with Compressors}

% % With the compressor, we consider the following conceptual compressed scheme associated with \cref{eqn:sweeping-process} given by
% % \begin{equation}
% %   \label{eqn:compressed-scheme}
% %   \bar{z}_{k+1} = P_{\mathscr{C}}[\bar{z}_k - h_{k+1} \varphi(t_k, \bar{z}_k, \vartheta)],
% % \end{equation}
% % where $\mathscr{C} \subseteq \mathbb{R}^m$ is a compact and convex subset, the compressor $\vartheta$ should be viewed as a functional parameter, and $\varphi(t,x,I) = \psi(t,x)$ for the identity map $I$. We then immediately obtain the following basic corollary:

% % \begin{corollary}
% %   \label{corollary:compressed-convergence}
% %   Let \cref{assumption:conv-for-iterates} hold. If $\varphi(t,x,\cdot)$ satisfies
% %   \begin{equation}
% %     \sup_{\{k:0\le t_k-t_n \le T\}} \left\| \sum_{\ell = n}^k h_\ell [\varphi(t_\ell,\bar{z}_\ell,\phi) - \psi(t_\ell,\bar{z}_\ell)] \right\| \to 0,
% %   \end{equation}
% %   the conclusion of \cref{theorem:general-func-V} holds for \cref{eqn:compressed-scheme}.
% % \end{corollary}

\begin{algorithm}[t]
  \textbf{Setup:} Each agent $i$ shares a common parameter $x^{-1} = x^0 \in C$, and applies a compressor $\vartheta_i$. Set step sizes $\{\lambda_k \ge 0, \alpha_k > 0\}$ and $k=0$. \\
  \textbf{Steps: (execute until a stopping criterion is satisfied)} \\
  1. Each agent $i$ obtains a noisy sample $g_i^k$ from $\nabla f_i(x^k + \lambda_k(x^k-x^{k-1})) + \xi_i^k$ and applies the compressor $\tilde{g}_i^k = \vartheta_i(g_i^k)$. \\
  2. Agents transmit the compressed gradients $\tilde{g}_i^k$ to the server. \\
  3. The server aggregates the compressed gradients and update the parameter by
  \begin{equation}
    \label{eqn:iterate-DPCGD}
    x^{k+1} = P_C\left[x^k - \frac{\alpha_k}{n} \sum_{i=1}^n \tilde{g}_i^k \right]. 
  \end{equation}
  4. The server send $x^{k+1}$ and the agents update $x^k \leftarrow x^{k+1}$. \\
  5. Set $k \leftarrow k+1$ and go back to step 1.
  \caption{The distributed projected compressed gradient descent (DPCGD) method.}
  \label{alg:DPCGD}
\end{algorithm}

% % Following the procedure in \cref{fig:conv-plot} and utilizing \cref{corollary:compressed-convergence} and the result (cf. \cref{corollary:conv-proj-Nesterov-form2}) in Example 2, Section \ref{sec:appl},

In this subsection, we propose a distributed projected compressed gradient descent (DPCGD) method for solving the following distributed optimization problem: $\min_{x\in C} f(x)$, for $f:= \frac{1}{n} \sum_{i=1}^n f_i$, where $f_i$ is the local private function of agent $i$. The algorithm is presented in \cref{alg:DPCGD}. The following assumption is required to show its convergence.

\begin{assumption}
  \label{asp:distributed}
  The following conditions hold:
  \begin{itemize}
    \item The set $C$ is compact and convex;
    \item Each $f_i$ is differentiable and its gradients are Lipschitz-continuous on $C$;
    \item Each $\vartheta_i$ satisfies \cref{eqn:unbiased-compressor} with respect to $\mu_i$;
    \item Given the KKT point set $\mathcal{L}$, $f(\mathcal{L})$ has empty interior;
    \item The step sizes are nonnegative and satisfy
    \begin{equation}
      \sum_{k=1}^\infty \alpha_k = \infty, \ \sum_{k=1}^\infty \alpha_k^2 < \infty, \ \lim_{k\to \infty} \lambda_k = 0;
    \end{equation}
    \item Denote the filtration by $\mathscr{F}_k := \sigma (x^\ell,\ell \le k)$. $\xi_i^k$ satisfies
    \begin{equation}
      \label{eqn:random-perturbation-DPCGD}
      \mathbb{E}[\xi_i^k | \mathscr{F}_k] = 0, \ \mathbb{E}[\|\xi_i^k\|^2 | \mathscr{F}_k] \le K_i (1 + \|x^k\|^2).
    \end{equation}
  \end{itemize}
\end{assumption}

\begin{theorem}
  \label{theorem:conv-DPCGD}
  Let \cref{asp:distributed} hold. The iterates $\{x^k\}$ generated by DPCGD converge to $\mathcal{L}$.
\end{theorem}
\begin{proof}
  It is sufficient to estimate the error between the compressed gradients and the raw gradients. To be specific, we need to bound $\zeta_i^k = \tilde{g}_i^k - \nabla f_i(x^k)$. Let $\nu_i^k$ stand for $\nabla f_i(x^k + \lambda_k(x^k - x^{k-1})) - \nabla f_i(x^k)$. Due to the almost sure boundedness of $\{x^k\}$ and $\lambda_k \to 0$, we have $\nu_i^k \to 0$ a.s. Moreover, we have $\mathbb{E}[\tilde{g}_i^k - g_i^k|\mathscr{F}_k] = 0$ and
  \begin{equation*}
    \mathbb{E}[\|\tilde{g}_i^k - g_i^k\|^2|\mathscr{F}_k] \le \mu_i \mathbb{E}[1+\|g_i^k\|^2|\mathscr{F}_k] \le K_i (1+\|x^k\|^2),
  \end{equation*}
  a.s. for some constant $K_i > 0$, because of the Lipschitz continuity of $\nabla f_i$, compactness of $C$ and the property of $\xi_i^k$ \cref{eqn:random-perturbation-DPCGD}. To summarize, \cref{eqn:iterate-DPCGD} can be written as
  \begin{equation*}
    x^{k+1} \! = \! P_C\!\left[x^k \! - \! \alpha_k \nabla f(x^k) \! - \! \frac{\alpha_k}{n} \! \sum_{i=1}^n (\tilde{g}_i^k \! - \! g_i^k \!+\! \nu_i^k \!+\! \xi_i^k)\right].
  \end{equation*}
  Via \cref{corollary:asymptotic-general-conv}, the proof is completed.
\end{proof}

%% file: sections/simulation.tex
\section{Numerical Simulation}
\label{sec:simulation}

In this section, we provide the outcomes of the numerical simulations, which serve to validate the theoretical convergence analysis presented in Section \ref{sec:appl}. Moreover, the simulation results for the proposed DPCGD are presented to show its effectiveness. The convergence behavior of the algorithms is demonstrated under randomized initial conditions, highlighting the stability of the methods and the effectivenesss of the developed theoretical results. 

\vspace{-8pt}
\subsection{Centralized Methods with Exact Gradients}
Specifically, we examine the standard PGD, PNAG with fixed parameters (FPNAG), PNAG with tuned time-dependent parameters (PNAG) and POGM. For PGD, we use vanishing step sizes as $h_k = 1/(k+1)$. For FPNAG, we let $\gamma = 0.1$ and $\mu = 0.5$. The parameters of PNAG and POGM follow the standard treatment as the unconstrained versions in the original articles (see \cite{1983NAG,2016OptimizedGradientDescent} for more detail). 

\begin{figure}[t]
  \centering
  \includegraphics[width=0.48\textwidth]{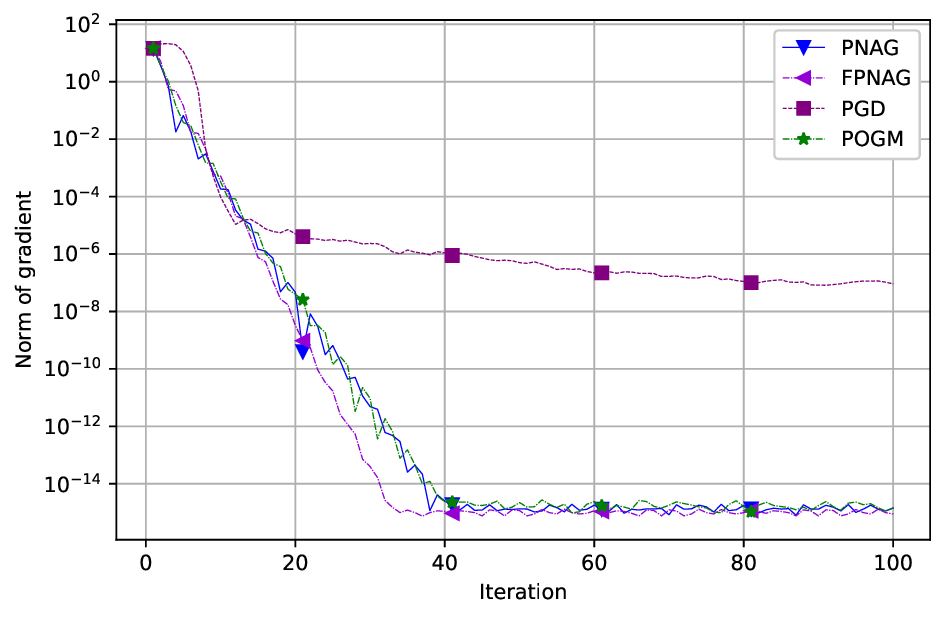}
  \caption{Results of constrained convex optimization.}
  \label{fig:convex-opt-without}
\end{figure}

We begin with a classical constrained convex optimization problem formulated as
\begin{equation}
  \label{eqn:convex-opt}
  \min_{x \in \mathcal{D}} f(x) = \frac{1}{2} \sum_{i=1}^{M} \|x - a_i\|^2,
\end{equation}
where $\mathcal{D} = [-1,1]^M$, $M = 10$ and each independent $a_i$ is randomly selected from a uniform distribution $U(-1,1)$. Since each $a_i \in [-1,1]$, it is straightforward to conclude that the critical point must lie in the interior of $\mathcal{D}$. Hence, the KKT point $x^*$ must satisfy $\nabla f(x^*) = 0$. This, indeed, matches the simulation result presented in Fig. \ref{fig:convex-opt-without}. Moreover, since both NAG and OGM are momentum-based methods, they inherently possess a convergence rate $O(1/k^2)$ compared to PGD with convergence rate $O(1/k)$. Further, the result implies that projection will not slow down the convergence rate, which can be deduced from the nonexpansiveness of projection in a geometric perspective. 

\begin{figure}
  \centering
  \includegraphics[width=0.48\textwidth]{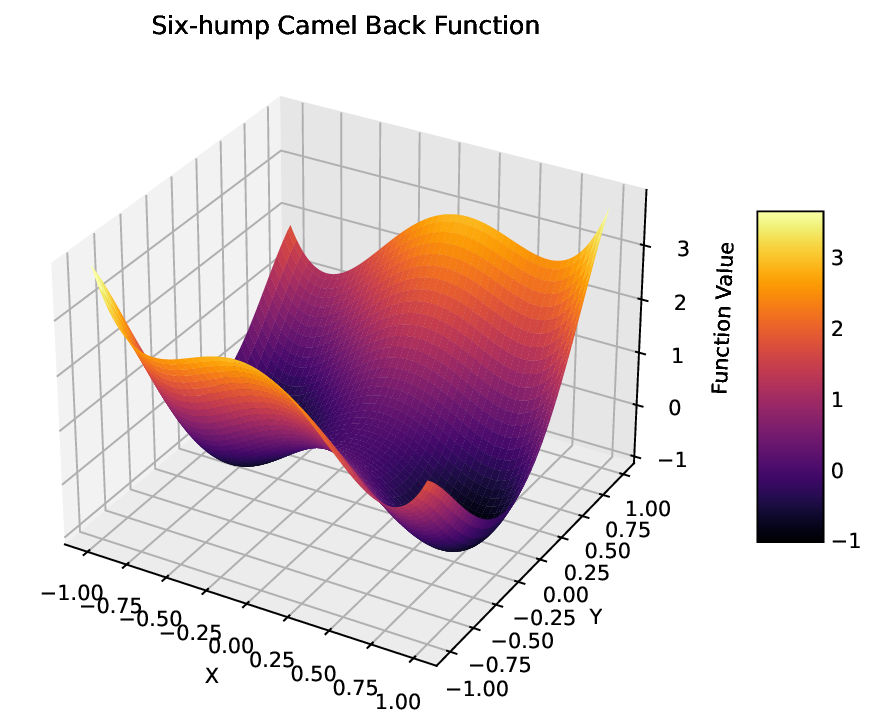}
  \caption{Six-hump camel back function.}
  \label{fig:sixh}
\end{figure}

\begin{figure}[t]
  \centering
  \includegraphics[width=0.48\textwidth]{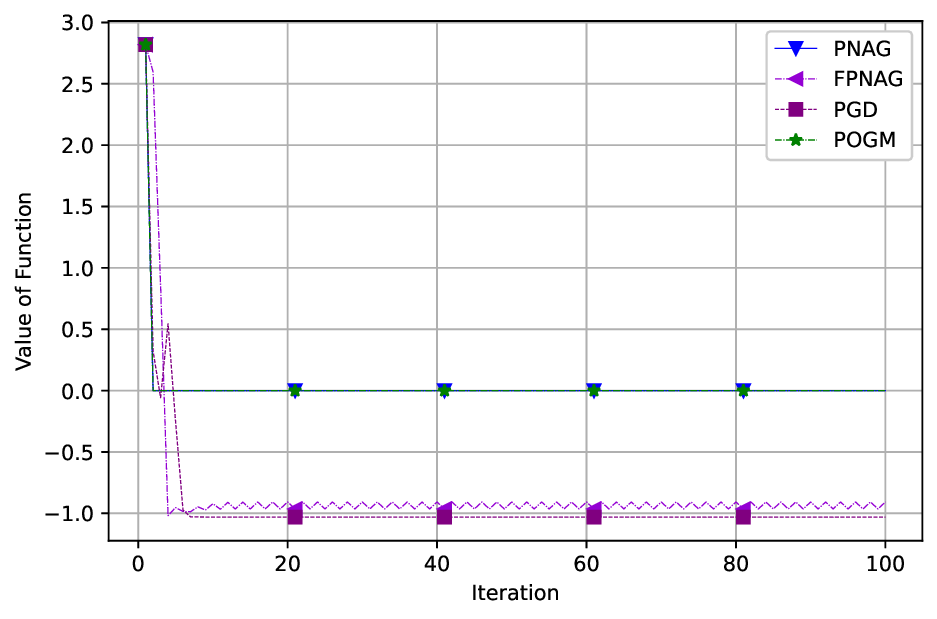}
  \caption{Results of optimizing the six-hump camel back function within the area $[-1,1]^2$.}
  \label{fig:sixh-opt-without}
\end{figure}

Next we consider a smooth nonconvex function called six-hump camel back function given by
\begin{equation}
  g(x,y) = (4-2.1x^2+x^{4/3})x^2 + xy + (-4+4y^2)y^2.
\end{equation}
As presented in Fig. \ref{fig:sixh}, the function has a global minimum $f^* = -1.0316$ for $(x^*,y^*) = \pm (0.0898,-0.7126)$ within the area $[-1,1]^2$. While it is in general NP-hard to find the global minimum for a constrained nonconvex optimization problem, PGD (FPNAG) succeeds to find this point (or oscillates around the neighborhood of the minimum) as presented in Fig. \ref{fig:sixh-opt-without}. Both POGM and PNAG fall into the trap of the saddle point. Especially, PGD is able to find a ``better'' local minimum (actually the best) than the other methods. This phenomenon is quite interesting since the standard gradient-based method outperforms momentum-based methods in both stability and final precision. Such a result indicates that the projected method may possess different convergence behavior from the unconstrained counterpart.

\begin{figure}[t]
  \centering
  \includegraphics[width=0.48\textwidth]{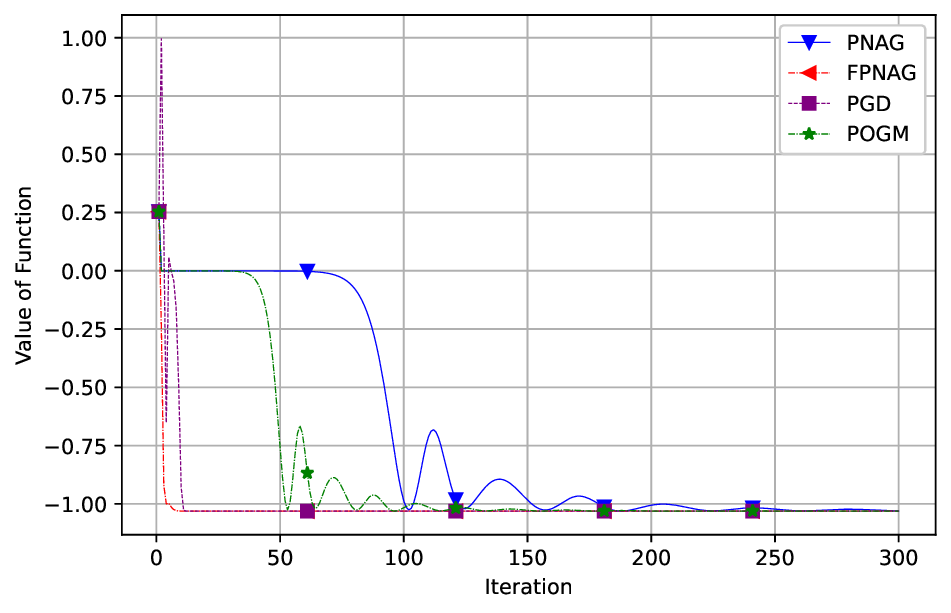}
  \caption{Results of optimizing the six-hump camel back function with methods incorporating random perturbations.}
  \label{fig:sixh-opt-with}
\end{figure}

Next, we discuss the effects of adding random perturbations to gradients in the iterations. The perturbation $\xi_k$ is a Gaussian stochastic vector with distribution $\mathcal{N}(0, \epsilon)$ for $\epsilon=0.001$. The corresponding stochastic scheme is modified at gradients ($\nabla f(y_k) \to \nabla f(y_k) + \xi_k$) and the learning rate factor before the gradient term ($\gamma \to \gamma/k$). The results are demonstrated in \ref{fig:sixh-opt-with}. It is clear that both PNAG and POGM benefit from the perturbation as for optimizing the six-hump camel back function as presented in Fig. \ref{fig:sixh-opt-with} compared with Fig. \ref{fig:sixh-opt-without}. To summarize, it could be beneficial to add random perturbations when the current local minimum is not ``good'' enough (a saddle point).

\vspace{-8pt}
\subsection{The Distributed Scheme with Compressed Gradients}
In this subsection, we apply a uniform random-vector compressor \cite{2024DistSubgradient} $\vartheta$, which is a simple extension of the scalar version by element-wise operation. To be specific, there exists some integer $\ell$ for any $x\in \mathbb{R}$ satisfying $\ell \le x < \ell + 1$. For a $b$-bit compressor, $x$ falls in $[\tau_i, \tau_{i+1})$, where $\tau_i= \ell + i \cdot 2^{-b}$ for $0 \le i \le 2^b$. Denoting the compressed random element by $q = \vartheta(x)$, we associate $x$ with $\tau_i$ or $\tau_{i+1}$ via
\begin{equation*}
  \mathrm{P}(q = \tau_{i+1}|x) = 2^b (x-\tau_i), \ \mathrm{P}(q =\tau_{i+1}|x) = 2^b (\tau_{i+1} - x),
\end{equation*}
which indicates that $\mathbb{E}[q|x] = x$, $\text{Var}(q) \le 4^{-b}$ and $\vartheta$ is an unbiased compressor with uniformly bounded variance. 

Consider a problem of power allocation for a wireless network composed of $N = 4$ sources and a central destination. We assume that the signal received by the destination is corrupted by an additive white Gaussian noise (AWGN) of variance $\sigma^2$ and the interference produced by the other sources. Denote by $A_i$ the channel gain between source $i$ and the destination and by $p_i$ the transmission power of source $i$. Therefore, we obtain the signal to interference-plus-noise ratio expressed by $A_ip_i / (\sigma^2 + \sum_{j\neq i} A_j p_j)$. We consider, for example, each transmitter uses a QPSK modulation and the corresponding bit error probability $F_i$ for transmitter $i$ is 
\begin{equation}
  F_i = Q\left( \sqrt{ \frac{A_ip_i}{\sigma^2 + \sum_{j\neq i} A_j p_j} } \right),
\end{equation}
where $Q(x) = \frac{1}{\sqrt{2\pi}} \int_x^\infty e^{-t^2/2} dt$. The objective is to minimize the weighted sum of bit error probabilities, i.e., 
\begin{equation}
  \min_{p \in \mathscr{C}} F(p) := \sum_{i=1}^N \gamma_i F_i(p),
\end{equation}
where $\gamma_i$ is the weight accessible to transmitter $i$ only, $p = (p_1,p_2,\dots,p_N)$, $\mathscr{C} = \{p: 0 < p_{\text{min}} \le p_i \le p_{\text{max}}, \forall i = 1,2,\dots, N\}$, and $p_{\text{min}}$ ($p_{\text{max}}$) is the minimum (maximum) transmission power of transmitter $i$. 

It is clear that the above optimization problem is nonconvex, the objective function is differentiable and has Lipschitz-continuous gradients on $\mathscr{C}$, and $\mathscr{C}$ is compact and convex. We can apply DPCGD to solve the problem. $\{\gamma_i\}$ are set to be $[0.4, 0.3, 0.2, 0.1]$, $\sigma^2 = 0.1$ and $A = [2, 5/3, 4/3, 1]$. The transmission power has limitations $p_{\text{min}} = 0.5$ and $p_{\text{max}} = 10$ (the unit can be ``Watt'' in practice). The step sizes are $\alpha_k = 100/k$ and $\lambda_k = 1/(1+\log(k))$. 

\begin{figure}[t]
  \centering
  \includegraphics[width=0.48\textwidth]{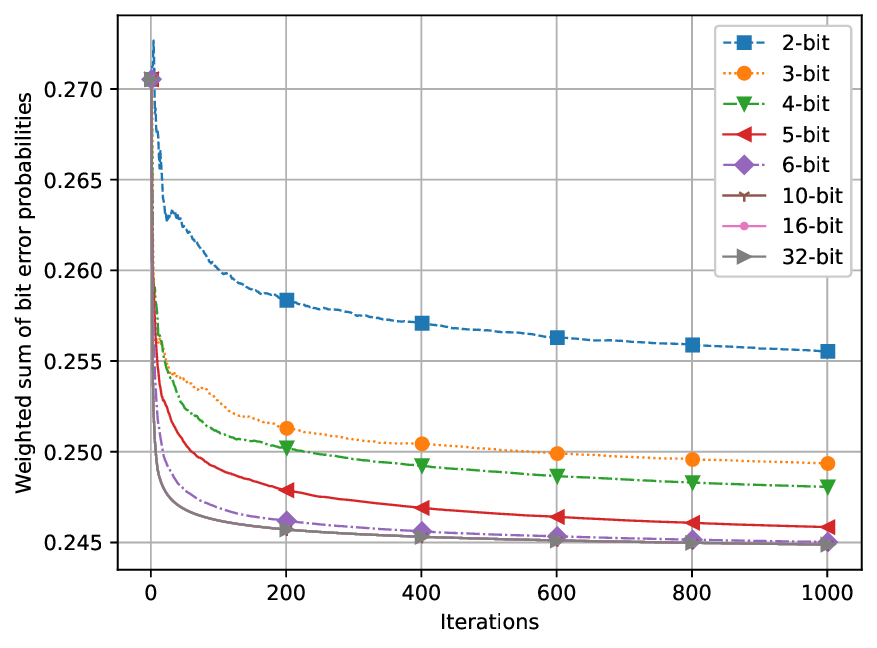}
  \caption{Weighted sum of bit error probabilities compression for deterministic channels with respect to different compression bits, averaged with respect to 30 Monte-Carlo runs. The 32-bit compressor is regarded as the ground truth.}
  \label{fig:quantization}
\end{figure}

\begin{figure}[t]
  \centering
  \includegraphics[width=0.48\textwidth]{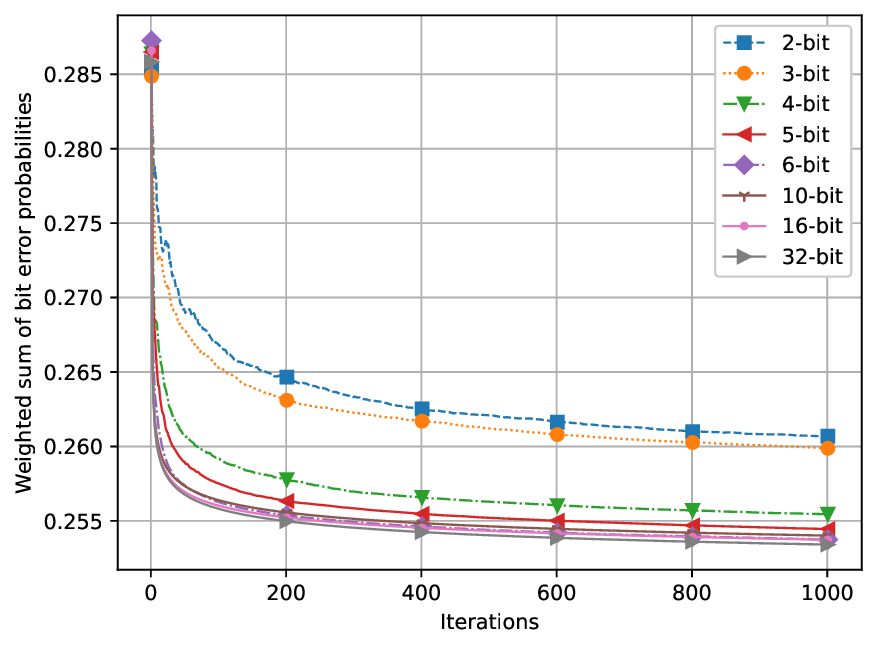}
  \caption{Weighted sum of bit error probabilities compression for random channels with respect to different compression bits, averaged with respect to 50 Monte-Carlo runs.}
  \label{fig:quantization-randomized}
\end{figure}

We examine the impact of varying bits on the performance of a compressor, as illustrated in Fig. \ref{fig:quantization}. The figure indicates that the rate of convergence within the scheme is adversely affected by the compressor, with the detrimental effect on convergence rate being inversely proportional to the number of bits allocated for compression. As the bit increases, the negative impact on the convergence rate is observed to decrease. Furthermore, the figure reveals that a 6-bit compression scheme is adequate to achieve a satisfactory rate of convergence, while simultaneously leading to a significant reduction in transmission data overhead.

Finally, we provide numerical results where the channel gains are random and time-varying. We let the channel gains be independently sampled from a uniform distribution $\mathcal{U}(0.5,1.5)$ such that $\mathbb{E}[A_i] = 1$ for all $i$. As shown in Fig. \ref{fig:quantization-randomized}, the trajectories against different bits are averaged based on 50 Monte-Carlo runs and the performance is comparable to the case of deterministic channels, which indicates that DPCGD is able to handle randomized channels. 

%% file: sections/conclusion.tex
\vspace{-6pt}
\section{Conclusion}
\label{sec:conclusion}

In this article, we have explored constrained compressed optimization, which is the fundamental problem of many low-bit resolution applications, through the lens of dynamical systems theory. By establishing the connection between a PSP and its Euler discretization, we have obtained theoretical results similar to counterpart results of unconstrained dynamics. Notably, we have developed a novel framework for convergence analysis that transcends traditional numerical methods. Several examples of convergence analysis of proximal gradient methods have been provided to demonstrate the effectivenesss of the framework. In the future, we plan to focus on more strigent constraints (one-bit signal for example) and decentralized problems on time-varying graphs. 

%% file: main.bbl
\begin{thebibliography}{10}
\providecommand{\url}[1]{#1}
\csname url@samestyle\endcsname
\providecommand{\newblock}{\relax}
\providecommand{\bibinfo}[2]{#2}
\providecommand{\BIBentrySTDinterwordspacing}{\spaceskip=0pt\relax}
\providecommand{\BIBentryALTinterwordstretchfactor}{4}
\providecommand{\BIBentryALTinterwordspacing}{\spaceskip=\fontdimen2\font plus
\BIBentryALTinterwordstretchfactor\fontdimen3\font minus
  \fontdimen4\font\relax}
\providecommand{\BIBforeignlanguage}[2]{{%
\expandafter\ifx\csname l@#1\endcsname\relax
\typeout{** WARNING: IEEEtran.bst: No hyphenation pattern has been}%
\typeout{** loaded for the language `#1'. Using the pattern for}%
\typeout{** the default language instead.}%
\else
\language=\csname l@#1\endcsname
\fi
#2}}
\providecommand{\BIBdecl}{\relax}
\BIBdecl

\bibitem{2023LocalLinearConv}
T.~Vu, R.~Raich, and X.~Fu, ``On local linear convergence of projected gradient
  descent for unit-modulus least squares,'' \emph{IEEE Trans. Signal Process.},
  vol.~71, pp. 3883--3897, 2023.

\bibitem{2024VariancedReduced}
Z.~Yang, F.-q. Xia, K.~Tu, and M.-C. Yue, ``Variance reduced random relaxed
  projection method for constrained finite-sum minimization problems,''
  \emph{IEEE Trans. Signal Process.}, vol.~72, pp. 2188--2203, 2024.

\bibitem{2023StochasticCompositionalGradient}
S.~T. Thomdapu, H.~Vardhan, and K.~Rajawat, ``Stochastic compositional gradient
  descent under compositional constraints,'' \emph{IEEE Trans. Signal
  Process.}, vol.~71, pp. 1115--1127, 2023.

\bibitem{2024ProjGD}
J.~Li, W.~Cui, and X.~Zhang, ``Projected gradient descent for spectral
  compressed sensing via symmetric hankel factorization,'' \emph{IEEE Trans.
  Signal Process.}, vol.~72, pp. 1590--1606, 2024.

\bibitem{2024Adaptive}
A.~M. Subramaniam, A.~Magesh, and V.~V. Veeravalli, ``Adaptive step-size
  methods for compressed sgd with memory feedback,'' \emph{IEEE Trans. Signal
  Process.}, vol.~72, pp. 2394--2406, 2024.

\bibitem{2021CompressedGD}
S.~Khirirat, S.~Magnússon, and M.~Johansson, ``Compressed gradient methods
  with hessian-aided error compensation,'' \emph{IEEE Trans. Signal Process.},
  vol.~69, pp. 998--1011, 2021.

\bibitem{2023QuantizationforDecentralized}
R.~Nassif, S.~Vlaski, M.~Carpentiero, V.~Matta, M.~Antonini, and A.~H. Sayed,
  ``Quantization for decentralized learning under subspace constraints,''
  \emph{IEEE Trans. Signal Process.}, vol.~71, pp. 2320--2335, 2023.

\bibitem{2024PairwiseConstraints}
F.~Han, X.~Cao, and Y.~Gong, ``Decentralized stochastic optimization with
  pairwise constraints and variance reduction,'' \emph{IEEE Trans. Signal
  Process.}, vol.~72, pp. 1960--1973, 2024.

\bibitem{2022SampleBased}
Y.~Cui, Y.~Li, and C.~Ye, ``Sample-based and feature-based federated learning
  for unconstrained and constrained nonconvex optimization via mini-batch
  ssca,'' \emph{IEEE Trans. Signal Process.}, vol.~70, pp. 3832--3847, 2022.

\bibitem{1996DynSystApprochStocApprox}
M.~Benaim, ``A dynamical system approach to stochastic approximations,''
  \emph{SIAM J. Control Optim.}, vol.~34, no.~2, pp. 437--472, 1996.

\bibitem{2023DSADMM}
G.~França, D.~P. Robinson, and R.~Vidal, ``A nonsmooth dynamical systems
  perspective on accelerated extensions of admm,'' \emph{IEEE Trans. Autom.
  Control}, vol.~68, no.~5, pp. 2966--2978, May. 2023.

\bibitem{2024DSTV}
R.~Raveendran, A.~D. Mahindrakar, and U.~Vaidya, ``Dynamical system approach
  for time-varying constrained convex optimization problems,'' \emph{IEEE
  Trans. Autom. Control}, vol.~69, no.~6, pp. 3822--3834, Jun. 2024.

\bibitem{2022TVOLti}
G.~Bianchin, J.~Cortés, J.~I. Poveda, and E.~Dall’Anese, ``Time-varying
  optimization of lti systems via projected primal-dual gradient flows,''
  \emph{IEEE Trans. Control Netw. Syst.}, vol.~9, no.~1, pp. 474--486, Mar.
  2022.

\bibitem{2024FixedTime}
X.~Shi, X.~Xu, G.~Wen, and J.~Cao, ``Fixed-time gradient flows for solving
  constrained optimization: A unified approach,'' \emph{IEEE/CAA J. Autom.
  Sinica}, vol.~11, no.~8, pp. 1849--1864, Aug. 2024.

\bibitem{2023DynamicOptimComplementarity}
P.~Stechlinski, ``Dynamic optimization of complementarity systems,'' \emph{IEEE
  Trans. Autom. Control}, vol.~68, no.~2, pp. 1122--1129, Feb. 2023.

\bibitem{2024SecondOrderPrimalDual}
M.~Tao, L.~Guo, J.~Cao, and L.~Rutkowski, ``A second-order primal-dual dynamics
  for set constrained distributed optimization problems,'' \emph{IEEE Trans.
  Circuits Syst. {II}}, vol.~71, no.~3, pp. 1316--1320, Mar. 2024.

\bibitem{2016PerturbedSweepingProcess}
A.~M. Dalila Azzam-Laouir and L.~Thibault, ``On perturbed sweeping process,''
  \emph{Applicable Anal.}, vol.~95, no.~2, pp. 303--322, 2016.

\bibitem{2019OptimalcontrolPSP}
T.~H. Cao and B.~Mordukhovich, ``Optimal control of a nonconvex perturbed
  sweeping process,'' \emph{J. Diff. Equations}, vol. 266, no.~2, pp.
  1003--1050, Jan. 2019.

\bibitem{2022PSPNonsmooth}
C.~Hermosilla and M.~Palladino, ``Optimal control of the sweeping process with
  a nonsmooth moving set,'' \emph{SIAM J. Control Optim.}, vol.~60, no.~5, pp.
  2811--2834, 2022.

\bibitem{2022GlobalAsympPSP}
L.~N. Wadippuli, I.~Gudoshnikov, and O.~Makarenkov, ``Global asymptotic
  stability of nonconvex sweeping processes,'' \emph{Discrete Contin. Dyn.
  Syst. Ser. B}, vol.~25, no.~3, pp. 1129--1139, Mar. 2020.

\bibitem{2011NumericalSchemeforPSP}
J.~Venel, ``A numerical scheme for a class of sweeping processes,''
  \emph{Numer. Math.}, vol. 118, no.~2, Jun. 2011.

\bibitem{2013ConvergenceOrder}
F.~Bernicot and J.~Venel, ``Convergence order of a numerical scheme for
  sweeping process,'' \emph{SIAM J. Control Optim.}, vol.~51, no.~4, pp.
  3075--3092, 2013.

\bibitem{2011FunctionalAnalysis}
H.~Brezis, \emph{Functional Analysis, Sobolev Spaces and Partial Differential
  Equations}, 1st~ed.\hskip 1em plus 0.5em minus 0.4em\relax Springer New York,
  NY, 2011.

\bibitem{2012LyapunovPair}
S.~Adly, A.~Hantoute, and M.~Théra, ``Nonsmooth lyapunov pairs for
  infinite-dimensional first-order differential inclusions,'' \emph{Nonlinear
  Anal. Theory Methods Appl.}, vol.~75, no.~3, pp. 985--1008, Feb. 2012.

\bibitem{2019LyapunovMonotone}
S.~Adly, A.~Hantoute, and B.~T. Nguyen, ``Lyapunov stability of differential
  inclusions involving prox-regular sets via maximal monotone operators,''
  \emph{J. Optim. Theory Appl.}, vol. 182, no.~3, pp. 906--934, Sept. 2019.

\bibitem{1984DifferentialInclusions}
J.-P. Aubin and A.~Cellina, \emph{Differential Inclusions: Set-Valued Maps and
  Viability Theory}, 1st~ed.\hskip 1em plus 0.5em minus 0.4em\relax Berlin,
  Heidelberg: Springer, 1984.

\bibitem{2018GlobalStability}
M.~Kamenskii, O.~Makarenkov, L.~{Niwanthi Wadippuli}, and P.~{Raynaud de
  Fitte}, ``Global stability of almost periodic solutions to monotone sweeping
  processes and their response to non-monotone perturbations,'' \emph{Nonlinear
  Anal. Hybrid Syst.}, vol.~30, pp. 213--224, 2018.

\bibitem{2006Equivalence}
B.~Brogliato, A.~Daniilidis, C.~Lemaréchal, and V.~Acary, ``On the equivalence
  between complementarity systems, projected systems and differential
  inclusions,'' \emph{Syst. Control Lett.}, vol.~55, no.~1, pp. 45--51, Jan.
  2006.

\bibitem{2012odeAndDS}
G.~Teschl, \emph{Ordinary Differential Equations and Dynamical Systems}.\hskip
  1em plus 0.5em minus 0.4em\relax American Mathematical Society, 2012.

\bibitem{2001ChainTransitivity}
M.~W. Hirsch, H.~L. Smith, and X.~Q. Zhao, ``Chain transitivity, attractivity,
  and strong repellors for semidynamical systems,'' \emph{J. Dyn. Diff.
  Equations}, vol.~13, no.~1, pp. 107--131, Jan. 2001.

\bibitem{2023StochasticApproxDynamics}
V.~S. Borkar, \emph{Stochastic Approximation: A Dynamical Systems Viewpoint},
  2nd~ed.\hskip 1em plus 0.5em minus 0.4em\relax Springer Singapore, 2023.

\bibitem{2005StochasticApproxDifferentialInclusion}
M.~Bena\"{\i}m, J.~Hofbauer, and S.~Sorin, ``Stochastic approximations and
  differential inclusions,'' \emph{SIAM J. Control Optim.}, vol.~44, no.~1, pp.
  328--348, Jan. 2005.

\bibitem{2016OptimizedGradientDescent}
D.~Kim and J.~A. Fessler, ``Optimized first-order methods for smooth convex
  minimization,'' \emph{Math. Program.}, vol. 159, no.~1, Sept. 2016.

\bibitem{2013ConvAMultiagentProjected}
P.~Bianchi and J.~Jakubowicz, ``Convergence of a multi-agent projected
  stochastic gradient algorithm for non-convex optimization,'' \emph{{IEEE}
  Trans. Autom. Control}, vol.~58, no.~2, pp. 391--405, Feb. 2013.

\bibitem{1983NAG}
Y.~Nesterov, ``A method of solving a convex programming problem with
  convergence rate {O}($1/k^2$),'' \emph{Sov. Math. Dokl.}, vol.~27, no.~2, pp.
  372--376, 1983.

\bibitem{2024DistSubgradient}
Z.~Xia, J.~Du, C.~Jiang, H.~V. Poor, Z.~Han, and Y.~Ren, ``Distributed
  subgradient method with random quantization and flexible weights: Convergence
  analysis,'' \emph{IEEE Trans. Cybern.}, vol.~54, no.~2, pp. 1223--1235, Feb.
  2024.

\end{thebibliography}
